\documentclass[11pt]{article}

\usepackage{amsfonts}
\usepackage{amsmath,amsthm,amscd,amssymb,mathrsfs,setspace}
\usepackage{latexsym,epsf,epsfig}
\usepackage{cite}
\usepackage{color}
\usepackage[hmargin=2.6cm,vmargin=2.6cm]{geometry}
\usepackage{hyperref}
\newcommand{\ds}{\displaystyle}

\theoremstyle{plain}
\newtheorem{theorem}{Theorem}[section]
\newtheorem{lemma}[theorem]{Lemma}

\theoremstyle{remark}
\newtheorem{definition}{Definition}
\newtheorem{remark}{Remark}[section]
\numberwithin{equation}{section}
\numberwithin{theorem}{section}
\numberwithin{remark}{section}
\numberwithin{assumption}{section}
\numberwithin{condition}{section}

\begin{document}

\title{Resonance and Periodic Solutions for \\  Harmonic Oscillators with General Forcing}
\author{%
\begin{tabular}[t]{c@{\extracolsep{6em}}c}
Isaac Benson\footnote{1000 Hilltop Dr., University of Maryland, Baltimore County, Baltimore, MD, 21250, USA} & Justin T. Webster$^*$   \\
\textit{UMBC} & \textit{UMBC}   \\
\textit{ibenson1@umbc.edu} & \textit{websterj@umbc.edu} \\
\end{tabular}%
}

\maketitle

\begin{abstract}
\noindent We discuss the notion of resonance, as well as the existence and uniqueness of periodic solutions for a forced simple harmonic oscillator. While this topic is elementary, and well-studied for sinusoidal forcing, this does not seem to be the case when the forcing function is general (perhaps discontinuous). Clear statements of theorems and proofs do not readily appear in standard textbooks or online. For that reason, we provide a characterization of resonant solutions, written in terms of the relationship between the forcing and natural frequencies, as well as a condition on a particular Fourier mode. While our discussions involve some notions from $L^2$-spaces, our proofs are elementary, using this  the variation of parameters formula; the main theorem and its proof should be readable by students who have completed a differential equations course and have some experience with analysis. We  provide several examples, and give various constructions of resonant solutions. Additionally, we connect our discussion  to  notions of resonance in systems of partial differential equations, including fluid-structure interactions and partially damped systems. 
\noindent  
\vskip.25cm

\noindent {\bf Keywords}: simple harmonic oscillator, resonance, periodic solutions, discontinuous forcing
\vskip.25cm
\noindent
{\em 2020 AMS}: 35C25, 34A36, 34C27, 70J35, 35B34
\vskip.4cm
\noindent Acknowledgements: The first author was partially supported by NSF DMS-1908033; the second author was partially supported by NSF DMS-1907620 and DMS-2307538. The authors thank Kaitlynn Lilly for her helpful comments which improved the exposition. 

\end{abstract}

\section{Introduction}
\subsection{Background}
The emergence of periodicity and the possibility of resonances in complex mechanical systems has been a topic of historical interest, continuing to the present day \cite{galdi,dowell}. This is particularly true in the field of fluid-structure interaction (FSI) systems \cite{sunny1,sunny2,book,srd,giusy}. Indeed, periodicity can be introduced into a multi-physics system through periodic inputs, e.g., \cite{galdi, sebastian} such as blood flow through an artery or organ \cite{sunny1,sunny2}; periodicity can also arise through bifurcations in what are deemed ``self-excitations", such as aeroelastic flutter, e.g., \cite{dowell,hhww}. In the former case, the entire circulatory system responds to the periodicity of the heartbeat; in the latter, the presence of a laminar airflow about an elastic structure may give rise to violent oscillations.
\vskip.5cm
\begin{center}
  \includegraphics[scale=.85]{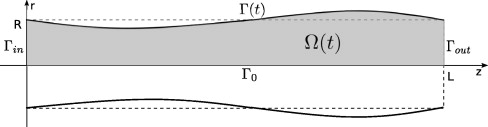}  \\   \quote{\small Fluid-structure interaction (FSI) in a time-evolving spatial domain modeling blood flow through a compliant artery \cite{sunny1}. The inlet boundary condition at $\Gamma_{in}$ is often taken to be a timer periodic function.}\end{center} \vskip.5cm
\begin{center}
\includegraphics[scale=.28]{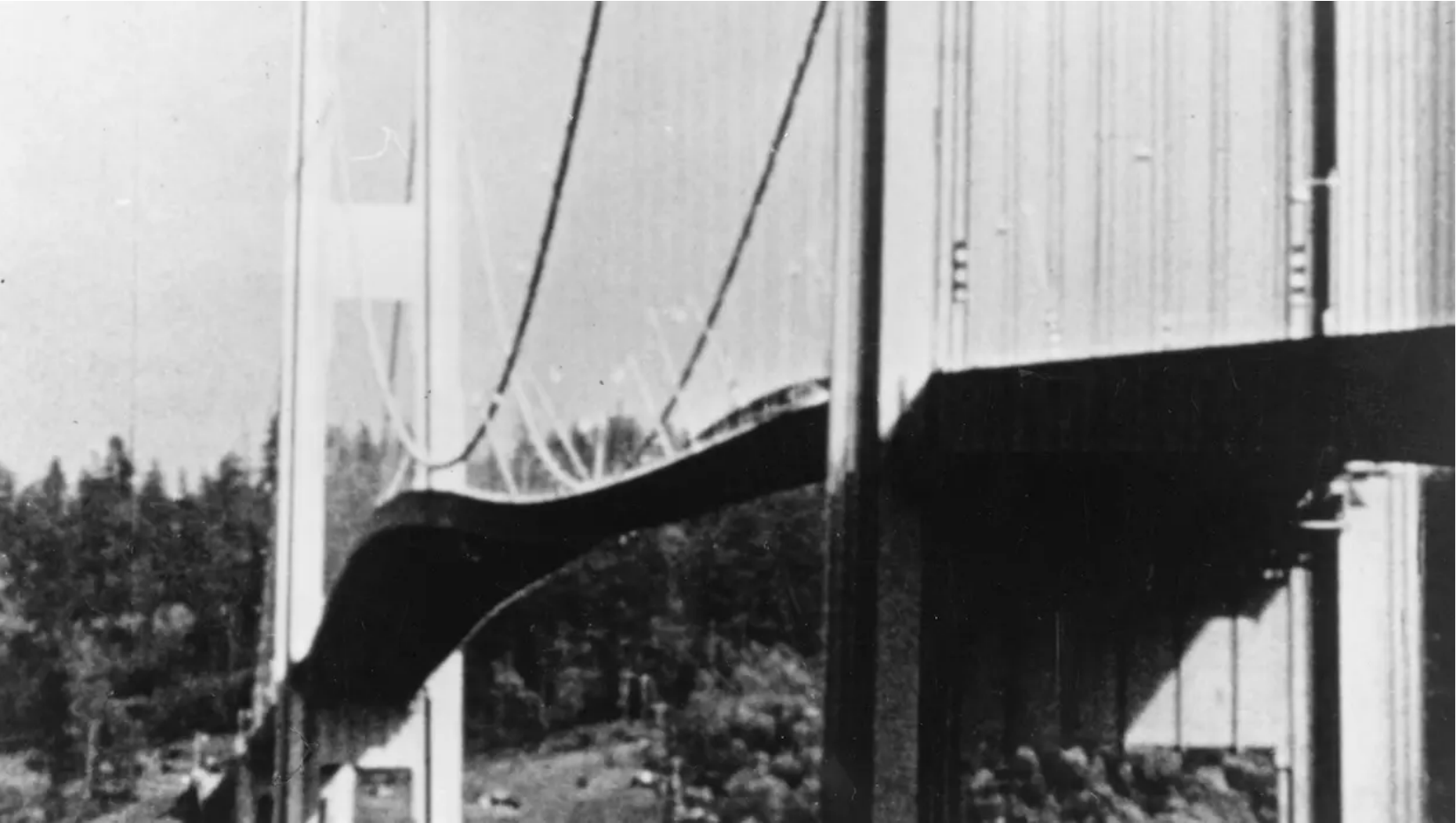} \hskip1.5cm \includegraphics[scale=.17]{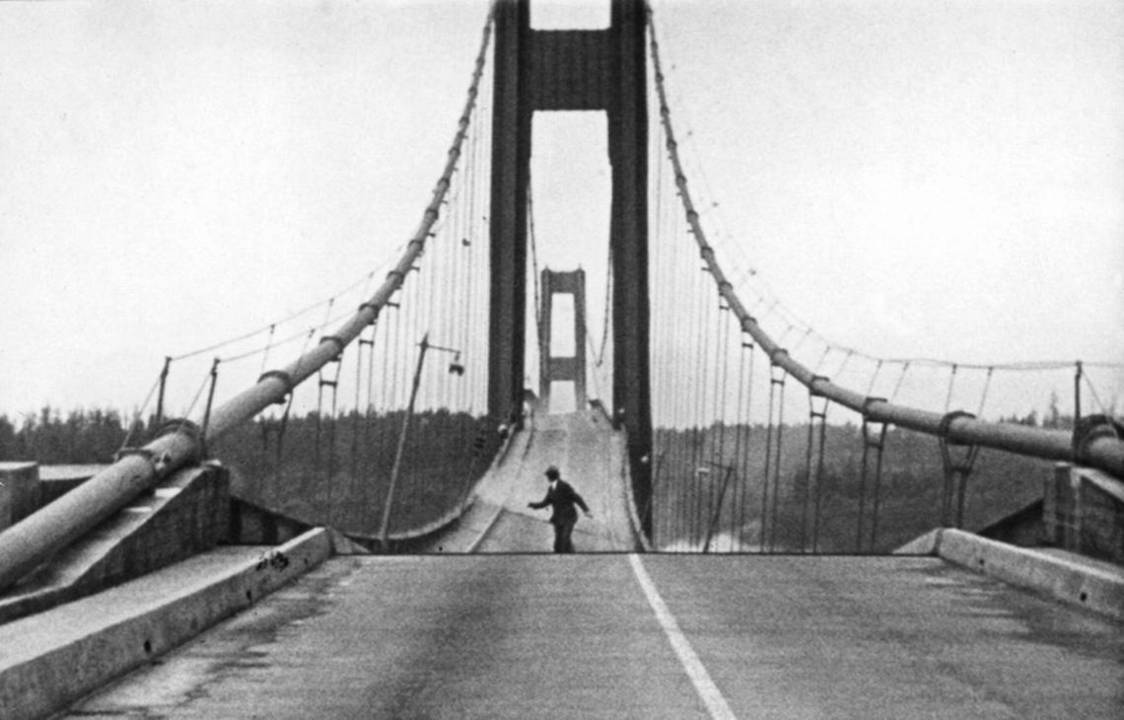} 
  \quote{\small Tacoma Narrows suspension bridge disaster, where periodicity emerged in this flow-structure system, resulting in dramatic structural failure.  [Image Credit: Library of Congress (L); Wikimedia Commons (R)]}
\end{center}
\vskip.5cm

\noindent Such considerations motivate sophisticated problems in the theory of partial differential equations (PDEs) and dynamical systems. These involve hyperbolic systems and hyperbolic-parabolic coupled systems, including the aforementioned FSIs \cite{galdi,msw}. Periodic solutions in the  context of challenging nonlinear FSI models is an emerging topic of interest \cite{sebastian,casanova,msw,giusy,srd}. While the theory of periodic solutions---existence and uniqueness thereof resulting from periodic forcing---is well-understood for parabolic problems (e.g., heat and Stokes equations) \cite{bostan,casanova,fuhrman}, constructing resonances in systems with a hyperbolic component \cite{brezis,coron,breziscoron} or, conversely, understanding periodic solutions, is more challenging \cite{galdi}. However, it is of fundamental importance to understand periodicity in both settings, if one is to rigorously analyze periodicity in coupled systems. 

\subsection{Motivating Problems} 
Working in conjunction with  undergraduate research students, we recently studied periodicity for a 1D heat-wave interaction on the spatial domain $x \in [0,L_1]\cup [L_1,L_2]$ with interface $x=\{L_1\}$. 
Let $v(x,t)$ denote a heat-type, i.e., fluid velocity, variable, and $u(x,t)$ denote a wave-type, i.e., structural displacement, variable. Then
the following PDE system is a scalar idealization of a fluid-structure interaction \cite{zuazua,galdi,msw}:
\begin{align}
\label{1Dsys}
	\dot v- v''&=f_F\;~~{\rm in}\;  (L_1,L_2)\times (0,T),
	\\
	\label{eq:solid1d}
	 \ddot u- u''&=f_S\;~~{\rm in}\; (0,L_1)\times (0,T),
	\\
	\dot u &=v\;~~{\rm on}\;  \{x=L_1\}\times (0,T),
	\\
	u'&= v'\;~~{\rm on}\; \{x=L_1\}\times (0,T),\\
	u & \equiv 0\;~~{\rm on}\; \{x=0\}\times (0,T),\\
		v & \equiv 0\;~~{\rm on}\;  \{x=L_2\}\times (0,T). \label{1Dsysend}
\end{align}
The notation above takes dots for time derivatives, and primes for spatial derivatives. 
The function $f_F$ is a ``finite-energy" time-periodic fluid forcing, distributed on $x \in (L_1,L_2)$, while $f_S$ is a similar forcing, distributed on $x \in (0,L_1)$. 

During a summer project, a group of motivated undergraduates worked on this problem spectrally and numerically. Their  aim was to work toward a resolution of the {\em open problem} \cite{galdi,zua1} of the possibility of resonance (and the existence and uniqueness of periodic solutions) to the 1D heat-wave system \eqref{1Dsys}--\eqref{1Dsysend}. As a byproduct, a foundational question arose: \begin{quote} Can we characterize resonances for the wave equation using Fourier-mode decomposition?\end{quote} 

With this background, we aimed  to construct  resonant solutions for 1D hyperbolic problems. We first considered time-periodic forcing $F_T(x, t)$ in the system in $w(x,t)$, as given by
\begin{equation}\label{wave} \ddot{w}-c^2 w''=F_T.\end{equation}
Here, $c>0$ represents a characteristic wave speed, and $T$ is the time-period of the spatially-distributed forcing, $F_T(x, t)$. 
The class of forcing arises from the physics of the problem, and {\em need not be continuous in time or space}---in practice, such a force need only satisfy {\em finite energy constraint(s)} such as being {\em time-space square-integrable}, denoted by $F_T \in L^2((0,T)\times \Omega)$. Through a spectral (or {\em modal}) approach, we may formally expand the solution as \begin{equation}\label{expansion}w(x,t) = \sum_{j} \phi_j(x)g_j(t),\end{equation} where the family $\{\phi_j\}$ represent a countable orthonormal basis---e.g., the eigenfunctions of $-\partial_x^2$  defined on  $x \in (0,L)$ with some set physically reasonable boundary conditions. Using a Galerkin procedure with the eigen-basis $\{\phi_j\}$ associated to positive eigenvalues $\{\lambda_j\}$, one may extract an ordinary differential equation (ODE) in each Fourier mode:
\begin{equation}\label{modal}
\ddot g_j+c^2\lambda_j g_j=f_j,
\end{equation}
where $\ds f_j(t)=\int_{0}^L F_T(x,t)\phi_j(x)dx$. If $F_T$ is chosen in an appropriate function class, these ODEs ``decouple" and can be solved independently, and reconstituted to form a solution representation for \eqref{wave} through the series \eqref{expansion}.  

To analyze resonant solutions in this framework, one may then work in the previously described modal context. In particular, we should have a broad understanding of periodic solutions for the simple harmonic oscillators acting in each mode (i.e., for each $j$ in \eqref{modal}) above. However, unlike classical forced oscillators which one  encounters in an elementary  ODE course, we must permit our forcing functions $f_j$ to be temporally square-integrable functions (denoted by $L^2(0,T)$) or perhaps only {\em bounded} functions (denoted by $L^{\infty}(0,T)$). In particular, this will  include the possibility of {\em discontinuous forcing}, and functions that exhibit {\em asymmetry on their period}. In particular, $f_j(t)$ need not be sinusoidal.

\begin{remark} We conclude this motivation section with a remark on the motivating  hyperbolic-parabolic PDE problem presented in  \eqref{1Dsys}. The theory described below (from \cite{galdi}) applies in that spatially 1D case, similarly as  in Section \ref{damped}. For the 1D system \eqref{1Dsys} \cite{zua1} (and higher dimensional versions \cite{zuazua}), homogeneous solutions are not uniformly {\em exponentially stable} in the sense of the solution semigroup \cite{book}. The abstract theory of \cite{galdi} cannot, then, guarantee that for time-space square-integrable forcing functions $f_S$ and $f_F$, both of period $T$, there exists a unique periodic solution of period $T$. Hence, the possibility of resonance is not excluded in this case. This problem has been recently investigated in all dimensions in \cite{msw}, and resonance has been excluded for smoother (in time) forcing functions, connecting to the {\em strong stability} (``generic") result in \cite{galdi}.
\end{remark}

\subsection{Problem at Hand}
Motivated by the above discussion, we consider  the simple mass-spring system in the displacement variable $x(t)$ with $x: [0,\infty) \to \mathbb R$, modeled through Hooke's and Newton's Second laws by
\begin{equation}\label{oscillator} \begin{cases}
\ddot{x}+\omega_0^2x=f(t)  \\
x(0)=x_0;~~\dot{x}(0)=x_1.
\end{cases}
\end{equation}
We consider initial data $x_0,x_1 \in \mathbb R$, and the natural (angular) frequency $\omega_0$ is given by $\omega_0^2=k/m$, where $m,k>0$ are oscillator mass and spring constants. The forcing  function $f$ will be of the form of $f(t) = f_T(t)$, namely a periodic function of {\em minimal} period\footnote{different authors use varied terminology such as essential, minimal, fundamental; minimality excludes constant functions from this consideration} $T>0$, {\em which is not necessarily continuous}; the subscript $T$ serves to emphasize the $T$-periodicity of $f_T$. To each such period $T>0$ we associate a frequency $\omega$, obeying the relation $\omega=\frac{2\pi}{T}$.  In what follows, we will consider $T$ (and thus $\omega$) to be variable, to be specified in the hypotheses of the main theorem. 
\begin{remark} In this note, we use the convention that the natural angular frequency is denoted $\omega_0$ and has units of $\text{rad}/\text{sec}$, in contrast to the so called rotational frequency $\nu = \frac{\omega_0}{2\pi}$, which has units of $\text{cycles}/\text{sec}$ ($\text{Hz}$). In this notation, the period $T=\frac{1}{\nu}=\frac{2\pi}{\omega_0}$. \end{remark}

Of course, sinusoidal forcing such as $f_{2\pi/\omega}(t) = A_0\cos(\omega t)$ is a well-studied and exposited  topic in elementary ODEs, see, e.g., \cite{farlow,BdP} and also the recent works \cite{gs,willms}. In this case, existence and uniqueness of solutions (and conditions on their periodicity) is classical, whenever $\omega \neq \omega_0$; the criterion for traditional {\em resonance}  is clear \cite{farlow}: when $\omega =\omega_0$ one obtains a so called {\em quasi-periodic solution} which grows unboundedly as $t \to \infty$. In this case, the {\em particular solution} associated to the forcing lies in the span of the fundamental set, $\{\sin(\omega_0t),\cos(\omega_0t)\}$. Using one's preferred solution method, the {\em resonant} particular solution  takes the form 
$$x_p(t) = t[C_1\sin(\omega_0t)+C_2\cos(\omega_0t)],$$ 
where the constants $C_1,C_2 \in \mathbb R$ are determined by the initial conditions, $\omega_0$, and the value $A_0$. We will elaborate on this in Section \ref{classical}. 

Upon reviewing the literature, we were surprised not to find a clear and detailed discussion of {\em resonant} behaviors for the general case of \eqref{oscillator}, when $f_T$ is periodic, but not sinusoidal. Although the topic of resonance is indeed classical, and covered in many undergraduate differential equations, mechanics, and engineering courses, the caveat seems to be that periodic forcing is almost always taken to be smooth, and, typically, sinusoidal. This motivated us to seek  a complete characterization of solutions to \eqref{oscillator}---when we have a periodic response to a periodic, yet not necessarily sinusoidal, forcing $f_T$.  Preliminary ODE solves, reproduced here, produced behaviors which are somewhat different than the classical sinusoidal case. We also mention \cite{ortega}, which focuses on Littlewood's problem; in the introduction there, an informal discussion of resonant conditions akin to our main result is made without further reference or proof.

In the present note, {\em we address existence and uniqueness of periodic solutions with general periodic forcing}. Along the way, we will provide elucidation of resonant scenarios. In practice we are motivated to consider $f_T \in L^{2}(0,T)$, however, we will see that we can accommodate functions $f_T$ for which the classical {\em variation of 
parameters formula} is valid. Namely, for $x_0=x_1=0$ in \eqref{oscillator}, we have a particular solution representation \cite{farlow}:
\begin{equation}\label{farlow} x(t)=\dfrac{1}{\omega_0}\int_0^tf(\tau)\sin\left(\omega_0(t-\tau)\right)d\tau.\end{equation}
This provides a solution  for less regular forcing functions, $f(t)$; namely, if $f$ is integrable (in some sense), then we can make sense of the particular solution $x(t)$ above as well as $\dot x(t)$ as continuous functions. Even if $x(t)$ does not satisfy \eqref{oscillator} classically, we can still determine the mapping $f(t) \mapsto x(t)$ and ascertain whether or not the associated response is (i) periodic, (ii) bounded but not periodic, or (iii) resonant (unbounded). The function $x(t)$ produced will of course satisfy the system \eqref{oscillator}, but in a generalized setting.

Our central criteria and main discussions here will avoid complexities associated with the point-wise (or smoother) convergence of Fourier series. Rather,  by focusing on the variation of parameters formula \eqref{varpar}, instead of Fourier expansions of $f(t)$ and $x(t)$, we are able to state and prove a criterion for periodicity in an entirely elementary fashion. Our main result below is Theorem \ref{th:main}; its statement, {\em as well as its proof}, should be accessible to any student who has completed a first course in ODEs and some analysis. The discussion about our result in the sequel should be readable to those who have had some experience in PDEs and Fourier series. \begin{remark} We will use some discussion of Fourier series to motivate our criteria and to cross-reference  our result. Indeed, Fourier expansions provide good intuition for what is happening at the level of the forcing, when $f_T \in L^2(0,T)$. Yet, we will avoid {\em proving} anything which requires a more advanced result on the convergence (in any sense) of Fourier series.\end{remark}

\section{Sinusoidal Forcing}\label{classical}
In this section we review the notions of classical resonance due to sinusoidal forcing as well as the notion of beats (modulation). We also discuss their interrelation, and, more broadly,  the precise scenarios when a periodic forcing gives rise to a periodic solution. The question---even in this simple case---is non-trivial, and affected by the relationship between the frequency of the forcing $\omega$, the natural frequency $\omega_0$, as well as the presence of the initial conditions $x_0,x_1$. Finally, we give some brief discussion of the notion of damping (or dissipation) in the traditional forced oscillator to provide some connection with the theory in \cite{msw,mmsw,galdi}. 

\subsection{Forced Oscillators, Beats, and Resonance}
Let us consider a simple motivational case, where we explicitly take $f_{\frac{2\pi}{\omega}}(t) = A_0 \cos(\omega t)$. The parameter $\omega$ is free to take any values in $\mathbb R_+ \equiv (0,\infty)$. Namely, we  consider
\begin{equation}\label{oscillator**}\ddot{x}+\omega_0^2x=A_0\cos(\omega t),\end{equation}
for some $A_0 \ge 0$. 
The general solution to the homogeneous equation, i.e., taken with $A_0 = 0$, is 
\begin{equation} 
x_h(t) = c_1 \cos(\omega_0t)+c_2\sin(\omega_0t). 
\end{equation}
The constants $c_1$ and $c_2$ are determined in a particular case by specifying the initial conditions $x(0)=x_0$ and $\dot x(0)=x_1$. Indeed, the associated initial value problem  is {\em well-posed}, which is
to say that solutions exist, are unique, and depend continuously on the  data $x_0,~x_1$. 
Solutions to \eqref{oscillator} obey the  {\em energy identity} (which yields the continuous dependence property)
\begin{equation}
\omega_0^2|x(t)|^2+|\dot x(t)|^2 = \omega_0^2|x_0|^2+|x_1|^2,~~\forall t \in [0,T].
\end{equation}
The solution is uniquely expressed as
\begin{equation}\label{cause}x(t) =x_0\cos(\omega_0t)+\dfrac{x_1}{\omega_0}\sin(\omega_0t).\end{equation}

Returning to the inhomogeneous problem \eqref{oscillator}, with $f(t) = A_0\cos(\omega t)$, we have a particular solution \begin{equation} \label{above} x_p(t) = \dfrac{A_0}{(\omega_0^2-\omega^2)}[\cos(\omega t)-\cos(\omega_0t)]\end{equation}
corresponding to the initial conditions $x(0)=\dot x(0)=0$.
For all $\omega\neq \omega_0$, this function is clearly bounded, though the coefficient $$\dfrac{A_0}{(\omega_0-\omega)(\omega_0+\omega)}$$ amplifies the solution in relation to the proximity of $\omega$ to $\omega_0$. This is connected to the notion of {\em beats}, when $\omega \approx \omega_0$ but $\omega \neq \omega_0$.

When $\omega=\omega_0$, so that $f(t)=A_0\cos(\omega_0t)$, we observe that the solution formula in \eqref{above} becomes singular. However, in the limit $\omega \to \omega_0$ \cite{willms,gs}, one obtains the particular solution:
$$x_p(t)=\dfrac{A_0}{2\omega_0}t\sin(\omega_0t).$$
We note that this solution constitutes an {\em unbounded response} (as $t \to \infty$) to the {\em bounded forcing} $f$. Said differently, the mapping $f\mapsto x$ is unbounded:   the forcing $f \in L^{\infty}(0,\infty)$, while the solution $x(t)$ has a magnitude which becomes arbitrarily large. 
In this case, we obtain the {\em resonant envelope}~ $ \pm \dfrac{A_0}{2\omega_0} t$ for solutions,
which leads to the {\em amplification factor}
$$\mathbf A(t) = \dfrac{A_0}{2\omega_0}t,~\text{ as }~t \to \infty,$$  measuring the solution's growth in time, {\em independent of initial conditions}. If we consider $t=T=2\pi\omega_0^{-1}$, then we can precisely calculate the ``growth per cycle" associated the frequency $\omega=\omega_0$. Lastly, we note that this solution is {\em quasi-periodic} in the sense that $x_p$ is a periodic function scaled by the  factor $\mathbf A$, which is increasing in $t$: 
$$x(t) = \mathbf A(t)\sin(\omega_0t).$$

For an example of resonance, take the ODE $$\ddot{x}(t) +  x(t) = \sin(t).$$ In this case $\omega_0=\omega=1$. With $x(0) = \dot x(0) = 0$, the unique solution is $$x(t) = \frac{1}{2} \sin(t) - \frac{1}{2} t \cos(t),$$ which is graphed below.
\begin{center} \includegraphics[scale=0.28]{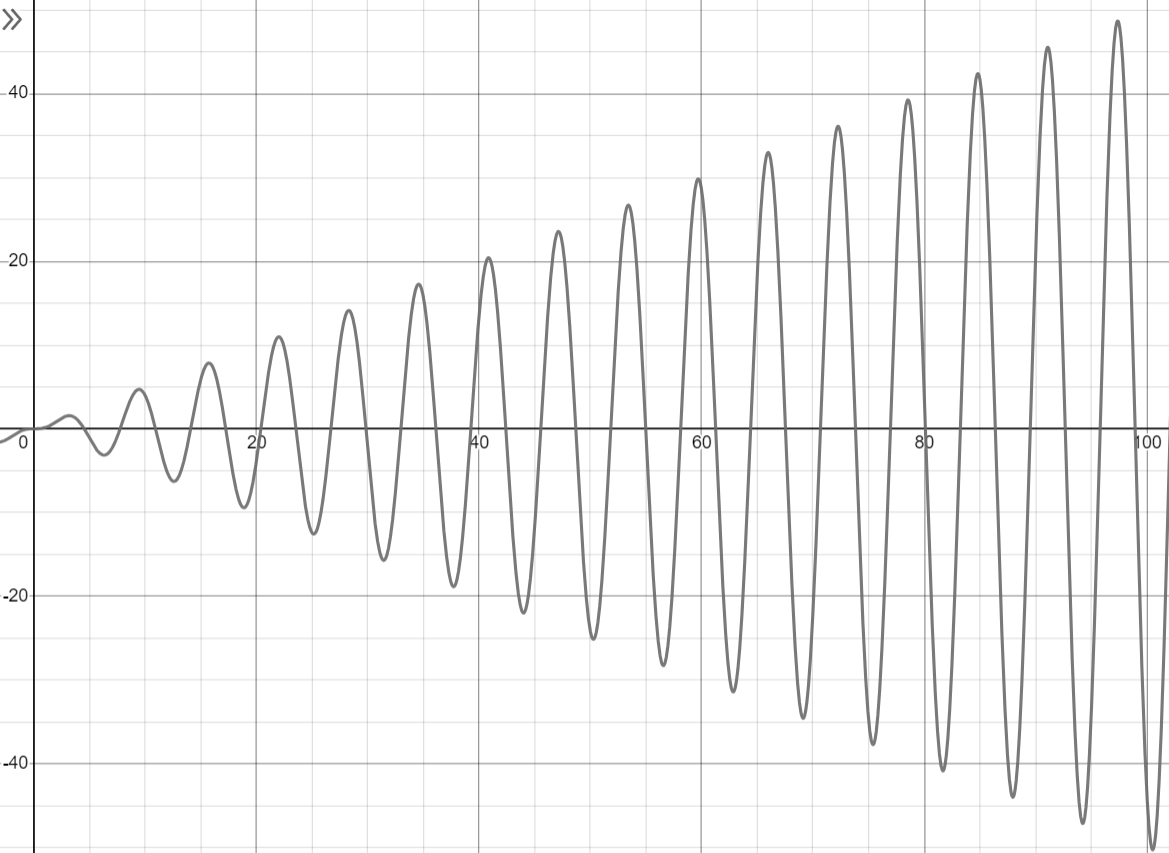} \end{center}
The response $x(t)$ is unbounded as $t \to \infty$. In that sense, the solution is not periodic (since the amplitude grows each cycle), but the quasi-period of the solution is $T=2\pi$. Now, if we modify the frequency of the forcing, so that $\omega =2$ resulting in $$\ddot{x}(t) +  x(t) = \sin(2t),$$ we note that the response is bounded. In this case, the forcing $\sin(2t) \notin \text{span}\{ \sin(\omega_0 t), \cos(\omega_0 t)\}$, and the solution is to the IVP with $x(0)=\dot x(0)=0$ is $$x(t) = \frac{2}{3} \sin(t) - \frac{1}{3} \sin(2t).$$ The solution is graphed below.
\begin{center} \includegraphics[scale=0.4]{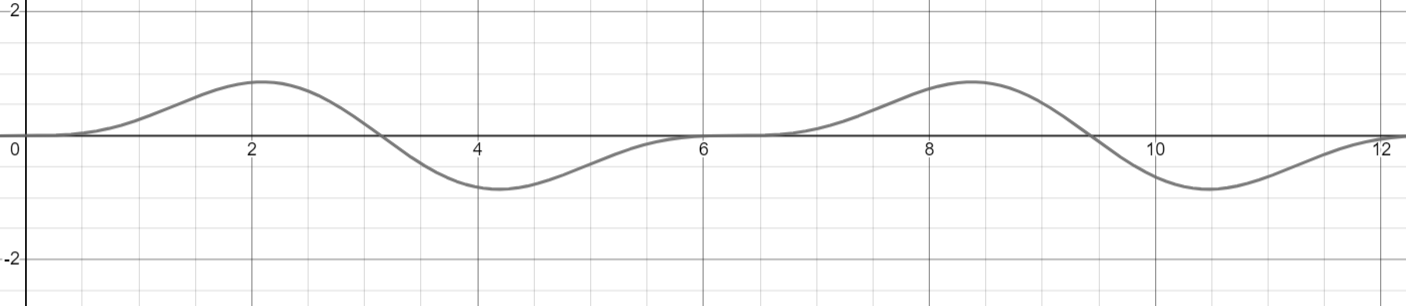} \end{center}

Let us now briefly describe the appearance of {\em beats} in the mass-spring system, and the relation to general periodic solutions. When $\omega \neq \omega_0$ we obtain a solution to \eqref{oscillator**} of the form:
\begin{equation}\label{gensol}
x(t) = x_h(t) + x_p(t)=c_1\cos(\omega_0t)+c_2\sin(\omega_0t)+C_1\cos(\omega t)+C_2\sin(\omega t).
\end{equation}
Such functions have interesting behaviors, and we mention some  non-trivial facts which will be relevant later---see \cite{periodic,periodic3}. Let us introduce the term {\em commensurate}, which occurs when two real numbers are rational multiples of one another.
\begin{theorem}\label{sumprod} Let $f: [0,T_1] \to \mathbb R$ and $g: [0,T_2] \to \mathbb R$ be  {\em continuous} periodic functions, of periods $T_1$ and $T_2$. Both functions can thus be periodically extended from $\mathbb R \to \mathbb R$, so let us identify $f$ and $g$ with those extensions. Then:
\begin{itemize}
\item A linear combination of $f$ and $g$ is periodic iff $T_1$ and $T_2$ are {\em commensurate}. 
\item The same holds, mutatis mutandis, for the product of $f$ and $g$.
\end{itemize}
In particular: if $T_1$ and $T_2$ are not rational multiples of each other, then there is no viable real period for a product or any non-trivial linear combination of $f$ and $g$.
\end{theorem}
Now we return to \eqref{gensol} with $A_0>0$. Let $\omega_0 = \frac{2\pi}{T_1}$ and $\omega = \frac{2\pi}{T_2}$. From the first fact above, when $\exists~m,n \in \mathbb Z$ so that $nT_1=mT_2$---and thus $\omega$ and $\omega_0$ are rational multiples of one another---the solution $x(t)$ is itself a periodic function. This periodicity is observed {\em independent of the values of $x_0$ and $x_1$}, though the resulting period may be affected by whether or not $c_1=c_2=0$ in \eqref{gensol}. If the initial conditions $\widetilde{x_0} =\dfrac{A_0}{\omega_0^2-\omega^2}$ and $\widetilde{x_1}=0$ are selected, then the unique solution to \eqref{oscillator**} is 
$$x(t)=\dfrac{A_0}{\omega_0^2-\omega^2}\cos(\omega t),$$ and thus has no contribution from the homogeneous solution of frequency $\omega_0$. (Note: In this particular case, the relationship between $\omega_0$ and $\omega$ actually does not matter, since the initial conditions eliminate the effect of $x_h$.) On the other hand, if $c_1c_2\neq 0$, then the period of the sum  in \eqref{gensol} is given by $T_3=nT_1=mT_2$. This can be made explicit. Namely, if we consider zero initial conditions $x_0=x_1=0$ for \eqref{oscillator**}, then the unique solution in \eqref{above} can be rewritten via a sum-to-product identity as
\begin{equation}\label{beat} x(t) = \dfrac{2A_0}{(\omega_0^2-\omega^2)}\sin\left(\dfrac{(\omega_0-\omega)}{2}t\right)\sin\left(\dfrac{(\omega_0+\omega))}{2}t\right)
\end{equation}
In this presentation, we can see that the above solution has {\em amplitude modulation}. Indeed, the function $x(t)$ will be periodic with a period related to $\dfrac{4\pi}{|\omega_0-\omega|}$; the first sinusoid {\em modulates} the amplitude of the second, which itself has  period $\dfrac{4\pi}{\omega_0+\omega}$. In this way, we can think of the {\em beat phenomenon} giving rise to {\em resonance} as $\omega \to \omega_0$. In this case, the amplitude modulation factor goes to $+\infty$ while the solution's period in \eqref{beat} also goes to $+\infty$; in other words, resonance occurs when $\omega$ reaches $\omega_0$, resulting in an infinitely long beat which approaches infinite amplitude. 
\vskip.2cm
\begin{center}
\includegraphics[scale=.3]{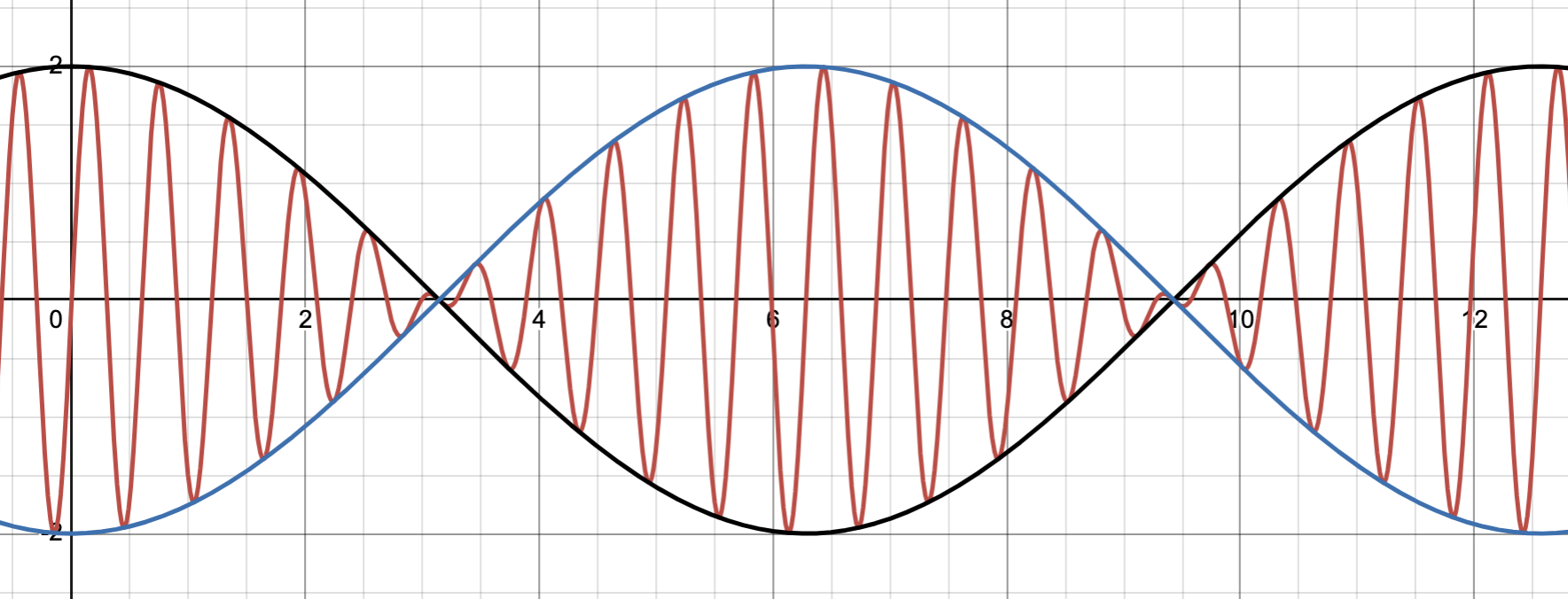}
\begin{quote} \small Above we see the sum of two sinusoids, $\sin(10t)$ and $\sin(11t)$, resulting in $\sin\left([10.5]t\right)$ modulated by $\sin\left([.5]t\right)$, resulting in the beat phenomenon. 
\end{quote}
\end{center}
\vskip.2cm
\begin{center}
\includegraphics[scale=.35]{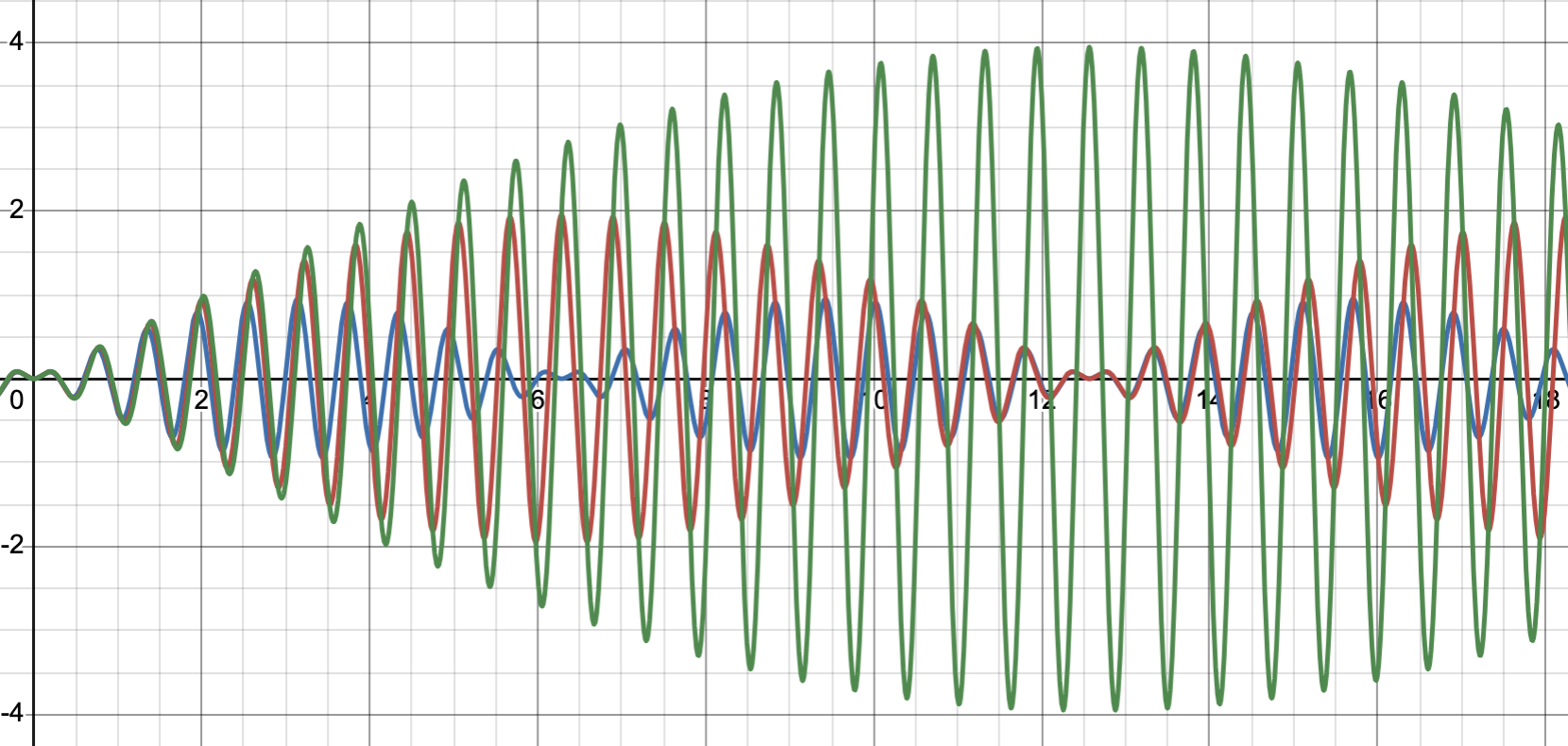}
\begin{quote} \small Above we observe the beats approaching resonance in \eqref{beat}, as the natural and forced frequencies begin to coalesce. In blue, we have $\omega_0=10$ and $\omega=11$; in red we have $\omega_0=10$ and $\omega=10.5$; finally, in green, we have $\omega_0=10$ and $\omega=10.25$. We note that the amplitude of \eqref{beat} grows as the modulated sinusoid spreads out.
\end{quote}
\end{center}

\subsection{Discussion: Existence and Uniqueness of Periodic Solutions}
Let us consider the unique solution to the full non-homogeneous (Cauchy) initial value problem in \eqref{oscillator}, with $f(t)=A\cos(\omega t)$ with $\omega \neq \omega_0$, represented by
\begin{equation}\label{thisone}
x(t) = x_h(t)+x_p(t) = x_0\cos(\omega_0t)+\dfrac{x_1}{\omega_0}\sin(\omega_0t)+  \dfrac{A_0}{(\omega_0^2-\omega^2)}[\cos(\omega t)-\cos(\omega_0t)].
\end{equation}
In the resonant case, when $\omega=\omega_0$, we also have the unique solution
\begin{equation}\label{resonantsol}
x(t) = x_h(t)+x_p(t) = x_0\cos(\omega_0t)+\dfrac{x_1}{\omega_0}\sin(\omega_0t)+\dfrac{A_0}{2\omega_0}t\sin(\omega_0t).
\end{equation}
We make some summative remarks for these solutions which allude to more general discussions in the next section.

Above, the forcing function $f(t)=A_0\cos(\omega t)$ and the representation of the solution $x_p$ satisfy a certain {\em compatibility condition}: the chosen particular solution $x_p$ contributes null initial conditions. This allows $x_h$ to subsume the role of (uniquely) achieving the initial conditions $x(0)=x_0$ and $\dot x(0)=x_1$. In more generality, for a given right hand side $f(t)$, one must be cognizant of the contributions that a particular solution, $x_p$, may make to the solution's initial conditions. These contributions may be relevant in determining periodicity, or in determining the specific period, of the resulting solution $x(t)$. 

We observed  cases:
\begin{itemize}
\item ($\omega_0=\omega$, $A_0>0$) The solution \eqref{resonantsol} is resonant for all initial conditions, and so no periodic solution exists. The homogeneous solution $x_h$ is periodic and the particular solution $x_p$ is quasi-periodic, both with period $\omega_0$.
\item ($\omega_0\neq \omega$, $A_0>0$) There are certain initial conditions, say $\widetilde{x_0}, ~\widetilde{x_1}$, which determine a periodic solution of period $T=2\pi\omega^{-1}$. If $\omega$ is a rational multiple of $\omega_0$, then for all initial conditions the solution is periodic, as the sum of commensurate sinusoids. When $\omega$ is not a rational multiple of $\omega_0$ (and again excepting the initial conditions $\widetilde{x_0}, ~\widetilde{x_1}$), the solution is not periodic of any period owing to Theorem \ref{sumprod}. 
\end{itemize}
Another way of looking at the ODE
$$\ddot{x}+\omega_0^2x=A_0\cos(\omega t)$$
is in terms of the existence and uniqueness of a periodic solution {\em with the same period as the forcing}, namely $T=\frac{2\pi}{\omega}$.
We can present the following dichotomy: 
\begin{itemize} 
\item In the non-resonant case ($\omega \neq \omega_0$), {\em there is at least one periodic solution of period $T=\frac{2\pi}{\omega}$} for which the associated  initial conditions $\widetilde{x_0}, ~\widetilde{x_1}$ can be recovered from \eqref{thisone},
namely  $\widetilde{x_0}=\frac{A_0}{\omega_0^2-\omega^2}$ and $\widetilde{x_1}=0$. However, uniqueness of a solution with this period depends on the particular relation between $\omega$ and $\omega_0$. (For instance, if $\omega$ is an integer multiple of $\omega_0$, the resulting solution will have infinitely many solutions with frequency $\omega$.) 
\item In the resonant case ($\omega=\omega_0$), there does not exist a periodic solution of period $T=\frac{2\pi}{\omega}$ for any initial conditions.
\item As we will discuss below in Section \ref{dampedsec}, in the case of a damped oscillator, given a sinusoidal forcing of frequency $\omega$, there will be exactly one set of initial conditions which provide the (unique) periodic solution of frequency $\omega$---all other solutions {\em will not be periodic}. 
\end{itemize}
This way of thinking is associated to the abstract question found in \cite{galdi}, namely: Given a forcing $f_T$ of period $T$, when can we ensure that there exists a unique periodic solution of period $T$? In this question, the initial conditions are thought of not as data, but as part of the  periodic response.

%

\subsection{Damped Oscillators and Some Comments on Dissipative Systems}\label{dampedsec}
To provide some additional context, we briefly consider damped oscillators. In particular, we will demonstrate why damped systems are easier to analyze from the point of view of periodic solutions. 

Consider a damping coefficient $d>0$ in the oscillator system:
\begin{equation}\label{oscillator***} \begin{cases}
\ddot{x}+d\dot{x}+\omega_0^2x=A_0\cos(\omega t)  \\
x(0)=x_0;~~\dot{x}(0)=x_1.
\end{cases}
\end{equation}
In this case, the general solution is represented by
\begin{equation}\label{damped}x(t)=x_h(t)+x_p(t) = c_1x_1(t)+c_2x_2(t)+C_1\cos(\omega t)+C_2\sin(\omega t).\end{equation}
The constants $c_1, c_2, C_1, C_2$ in \eqref{damped} are uniquely determined from $x_0,x_1,A_0,\omega_0, d$. 
The homogeneous solution $x_h$ obeys the damped energy identity:
$$\omega_0^2|x_h(t)|^2+|\dot x_h(t)|^2+2d\int_0^t|\dot x(s)|^2ds=\omega_0^2|x_0|^2+|x_1|^2.$$
While there are three possible forms of the homogeneous solution (depending on how the damping $d>0$ compares to the quantity $2\omega_0$), in each case,  $\exists~K,\gamma>0$ so that $$|x_h(t)| \le K(x_0,x_1)e^{-\gamma t},~~~t\to \infty.$$ The residual (or {\em steady-state}) solution is given by $x_p(t)=C_1\cos(\omega t)+C_2\sin(\omega t)$ for some $C_1$ and $C_2$, and represents the properly periodic component of the solution. We note that no resonance is possible in this scenario, but when the system is underdamped (so $0 < d < 2\omega_0$), we can maximize the  amplification coefficients $C_1$ and $C_2$. This occurs when the condition $\omega = \sqrt{k/m-d^2/4m^2}$ is exactly satisfied. 

To obtain a properly periodic solution from \eqref{damped}---namely, when the solution to the Cauchy problem is time-periodic, without a decaying component---we must find the unique initial conditions $\widetilde{x_0}$ and $\widetilde{x_1}$ which explicitly make $c_1=c_2=0$, i.e., which zero-out the transient component of the solution $x_h(t)$. By way of example, with $f(t)= A_0\cos(\omega t)$ as in \eqref{oscillator***}, a particular solution is given by
\begin{equation}\label{thisperiodic} x_p(t)= \dfrac{A_0(\omega_0^2-\omega^2)^2}{\omega_0^2-\omega^2+(\omega d)^2}\cos(\omega t) - \dfrac{A_0(\omega_0^2-\omega^2)}{\omega_0^2-\omega^2+(\omega d)^2}\sin(\omega t).\end{equation}
In this case, \begin{align} x_p(0)=~ \dfrac{A_0(\omega_0^2-\omega^2)^2}{\omega_0^2-\omega^2+(\omega d)^2};~~~  \dot x_p(0)=~  -\dfrac{\omega A_0(\omega_0^2-\omega^2)}{\omega_0^2-\omega^2+(\omega d)^2}.\end{align} Hence, to obtain a truly periodic solution, $x_h(0)$ and $\dot x_h(0)$ must be chosen (uniquely, in each particular case as determined by $d>0$) so that
\begin{equation}\label{specialIC}x_h(0)=-x_p(0)~~\text{ and }~~\dot x_h(0) = -\dot x_p(0).\end{equation}

Thus, the damped oscillator case may be compared to the non-resonant case in the previous section; namely, for \eqref{oscillator***}, the forcing uniquely determines the initial conditions which support a unique periodic solution of period $T=2\pi\omega^{-1}$, for any frequency $\omega \in \mathbb R$. If   initial conditions  do not satisfy the compatibility condition in \eqref{specialIC}, one will obtain a unique solution to the Cauchy problem, but it will not be periodic, due to the contribution from $x_h(t)$. Of course, for all $x_0, x_1$, we will observe that $$x(t) = x_h(t)+x_p(t) \to x_p(t), ~\text{ as }~t\to \infty,$$ which is to say  all solutions  converge to the unique periodic solution  \eqref{thisperiodic} associated to the special initial conditions in \eqref{specialIC}. This is precisely a demonstration of the abstract theory presented in \cite{galdi}.

\section{Statement of Main Result and Discussion}
\subsection{Mathematical Preliminaries}
We now discuss some basic concepts from ODEs and Fourier theory which will be relevant to the main theorem. We do not expand upon  these ideas at length, but  present the facts and notation we will use in our work below.

We will work with the space $L^2(0,T)$, where $T>0$. This will be taken (na\"ively) as the vector space of square integrable functions $f: (0,T) \to \mathbb R$. This  is a {\em real} Hilbert space, with inner product and norm defined by:
\begin{align}
(f,g)_{L^2(0,T)} \equiv &~\int_0^Tf(t)g(t)dt,~~\forall~~f,g \in L^2(0,T),\\[.2cm]
||f||^2_{L^2(0,T)} \equiv &~ (f,f)_{L^2(0,T)}. \label{norm}
\end{align}
Although we will not explicitly use facts concerning Fourier series (and their convergence), $L^2$ is a useful function space and will provide convenient notation for what follows. We will also use the notation $C^n([0,T])$ to denote the space of functions $f: [0,T]\to \mathbb R$ which are $n$-times continuously differentiable; differentiation at $x=0$ and $x=T$  is taken only from the right and left, respectively.
We do not assume any concepts from measure theory, but to be more precise, we can  define $L^2(0,T)$ as the completion of the smooth functions in \eqref{norm}. This is to say $L^2(0,T) \equiv [\overline{C^{\infty}(0,T)}]_{L^2(0,T)}$, and so we consider the $C^{\infty}(0,T)$ functions and their limit points in   $||\cdot||_{L^2(0,T)}$. Although we do not expound upon periodic extensions here, or the convergence of Fourier series, we  note from elementary PDEs that there is a connection in $L^2(0,T)$ to smoother periodic functions. Namely,  we can consider a boundary value problem in $s(t)$ for the eigenfunctions of the  differential operator $-\partial_t^2$ defined on $L^2(0,T)$:
\begin{align}
-\partial_t^2s(t)=&~\lambda s(t),~~t \in (0,T), \\[.2cm]
s(0)=s(&T),  ~~\dot s (0)=\dot s(T).
\end{align}
Using the theory of Sturm-Liouville (or the Spectral Theorem for positive, self-adjoint operators with compact resolvent), we have a countable family of orthonormal eigenfunctions, associated to a countable family of positive eigenvalues. These eigenfunctions form a basis for $L^2(0,T)$ and have the property that they are smooth and, in this case, periodic. We can express them  in two ways:
\begin{align*} s_n(t) =&~ a_n\sin\left(\dfrac{2\pi}{T}nt\right)+b_n\cos\left(\dfrac{2\pi}{T}nt\right),~~n \in \mathbb W\\[.1cm]
=& ~c_n\exp\left(\dfrac{2\pi i}{T}nt\right),~~n \in \mathbb Z,\end{align*}
where $a_n,  b_n,$ and $c_n$ are the appropriate Fourier coefficients.
One may use either representation, being careful to note the particular Fourier coefficients in each situation. Using sine and cosine brings the problem squarely in the context of the variation of parameters formula, described below. 
It will be convenient here to recall the orthogonality of the eigenfunctions on $[0,T]$. We recall that whenever $m, n \in \mathbb Z$ with $m \neq n$:
\begin{align}
\left (s_n(t),~s_m(t)\right)_{L^2(0,T)} = & ~0\\
\left (\cos\left(\dfrac{2\pi}{T}nt\right),~\cos\left(\dfrac{2\pi}{T}mt\right)\right)_{L^2(0,T)} = &~ 0\\
\left (\sin\left(\dfrac{2\pi}{T}nt\right),~\sin\left(\dfrac{2\pi}{T}mt\right)\right)_{L^2(0,T)} = & ~0.
\end{align}
Hence the families $$\Big\{\exp\left(\dfrac{2\pi i}{T}nt\right)\Big\}_{n \in \mathbb Z},~~\Big\{\sin(nx/2)\Big\}_{n=1}^{\infty} ~\text{ and }~ \Big\{\cos(nx/2)\Big\}_{n=0}^{\infty}$$ are orthogonal.

We now proceed with a discussion of the classical {\em variation of parameters formula}, based on the so called fundamental set ~$\big\{\sin(\omega_0t),~\cos(\omega_0t)\big\}$
for \eqref{oscillator}, taken with $f_T \equiv 0$. 
A particular solution to the non-homogeneous ODE system  in \eqref{oscillator} is then given by the variation of parameters (or Duhamel) formula
\begin{equation} \label{varpar} x(t)=\dfrac{\sin(\omega_0t)}{\omega_0}\int_0^t\cos(\omega_0\tau)f(\tau)d\tau-\dfrac{\cos(\omega_0t)}{\omega_0}\int_0^t\sin(\omega_0\tau)f(\tau)d\tau,\end{equation}
 taking into account the aforesaid fundamental set. 

In the context of this note, we will consider forcing functions which are {\em piecewise} continuous on their period, and periodic with some given period $T$, so $f=f_T$. We will restrict to functions which have only a {\em finite number of jump discontinuities} on $[0,T]$. The expression  \eqref{varpar} provides the unique (generalized) solution to \eqref{oscillator}: For all $t\ge 0$ 
\begin{equation}\label{formula}
x(t) = x_0\cos(\omega_0t)+\dfrac{x_1}{\omega_0}\sin(\omega_0t)+\dfrac{\sin(\omega_0t)}{\omega_0}\int_0^t\cos(\omega_0\tau)f(\tau)d\tau-\dfrac{\cos(\omega_0t)}{\omega_0}\int_0^t\sin(\omega_0\tau)f(\tau)d\tau,
\end{equation}

The interpretation of the solution, is of course in a generalized sense, since the function $x(\cdot )$ need not be $C^2([0,T])$. This depends, of course, on the regularity of $f(\cdot)$ on its period $0<t<T$. We do note that periodicity and piecewise continuity ensure $f(0)=f(T)$. Moreover, we can restrict the Cauchy problem to intervals of continuity for the given function $f$ and employ standard arguments for continuous forcing functions. On such an interval, say $[a,b]$, we have
\begin{equation}\label{inhomgeident}
\omega_0^2|x(t)|^2+|\dot x(t)|^2 = \omega_0^2|x_0|^2+|x_1|^2+2\int_0^tf(\tau)\dot x(\tau)d\tau,~~\forall t \in [a,b].
\end{equation}
From the standard Gronwall inequality, we obtain, for some $C>0$, that
\begin{equation}\label{gron} \omega_0^2|x(t)|^2+|\dot x(t)|^2 \le Ce^T\left[\omega_0^2|x_0|^2+|x_1|^2+\int_0^T|f(\tau)|^2d\tau\right],~~ t\in [a,b].\end{equation}
The system \eqref{oscillator} with $f:[0,\infty)\to \mathbb R$ bounded and piecewise continuous is well-posed. Existence (in a generalized sense) obtains from \eqref{formula}, which gives $x(\cdot)  \in C([0,T])$, with the first derivative being right continuous (we will say more below). Uniqueness follows from uniqueness of the homogeneous solution, by superposition. Continuous dependence (on the data $x_0, x_1, f_T$) holds---in line with \eqref{gron} above---but is a more subtle issue, owing to the fact that globally, on $[0,T]$, we allow $f$ to have jump discontinuities. 

We now provide some comments about the smoothness of $x(t)$.
First, one may suspect that jumps in $f$ may lead to the same in $\dot x$. However, a calculation with the Fundamental Theorem of Calculus for \eqref{varpar} yields:
\begin{align}
\dot x(t) =&~\dfrac{d}{dt}\left[\dfrac{\sin(\omega_0t)}{\omega_0}\int_0^t\cos(\omega_0\tau)f_T(\tau)d\tau-\dfrac{\cos(\omega_0t)}{\omega_0}\int_0^t\sin(\omega_0\tau)f_T(\tau)d\tau\right]\\
= &~\cos(\omega_0 t)\int_0^t\cos(\omega_0\tau)f_T(\tau)d\tau +\dfrac{1}{\omega_0}\sin(\omega_0t)\cos(\omega_0t)f_T(t) \\ & +\sin(\omega_0t)\int_0^t\sin(\omega_0 \tau)f_T(\tau)d\tau-\dfrac{1}{\omega_0}\cos(\omega_0 t)\sin(\omega_0t)f_T(t) \nonumber \\
=&~ \cos(\omega_0 t)\int_0^t\cos(\omega_0\tau)f_T(\tau)d\tau +\sin(\omega_0t)\int_0^t\sin(\omega_0 \tau)f_T(\tau)d\tau.
\end{align}
The above demonstrates cancellation of potentially discontinuous terms, and we observe proper continuity of $\dot x$. When $f$ has jump discontinuities, we can interpret the above solution in the context of distributions. In that case, when $f \in L^2(0,T)$ the formula above \eqref{formula} represents a so called {\em mild} or {\em generalized} solution \cite{pazy} which satisfies a time-integrated version of the system \eqref{oscillator}. 

The solution in  \eqref{formula} (with initial conditions), demonstrates the regularity that $x$ and $\dot x$ are continuous functions, and $\ddot x$ is continuous on intervals of continuity of $f$ (in particular, it is right continuous). Namely: $x(\cdot) \in C^1([0,T])\cap C^2_r([0,T))$, the derivatives at $t=0$ are of course only taken from the right. With $f \in L^2(0,T)$, the expressions in \eqref{inhomgeident} and \eqref{gron} hold for the solution given by \eqref{formula}. Moroever, with $f_T$ periodic of period $T$, the solution $x$ (as well as $\dot x$ and $\ddot x$) can be periodically extended to $\mathbb R_+ \equiv (0,\infty)$, maintaining the same properties.

\begin{remark}
The particular solution above (corresponding to $x_0=x_1=0$) can also be obtained through the proper Duhamel procedure, yielding the convolution  representation
$$x(t)=\dfrac{1}{\omega_0}\int_0^tf(\tau)\sin\left(\omega_0(t-\tau)\right)d\tau;~~\dot x(t)=\int_0^tf(\tau)\cos\left(\omega_0(t-\tau)\right)d\tau.$$
In this framework, we also have the acceleration, which is  in $L^2(0,T)$ when $f \in L^2(0,T)$:
$$\ddot{x}(t) = f(t)-\omega_0\int_0^tf(\tau)\sin(\omega_0(t-\tau))d\tau.$$
\end{remark}

\subsection{Statement of the Main Result}
We recall the mass-spring system, with zero Cauchy data: 
\begin{equation}\label{sys1*}\begin{cases} 
      \ddot x(t) + \omega_0^2 x(t) = f(t) \\
      x(0) = 0; ~~\dot x(0) = 0.
   \end{cases}
\end{equation}
We recall the {\em standing notation} that $\omega_0^2 = k/m \in \mathbb{R}_+$ and we denote the natural (unforced) period of the system by $T_1 = \frac{2\pi}{\omega_0}>0$. We take, as {\em standing assumptions}, that the forcing $f:[0,\infty)\rightarrow\mathbb{R}$ is a piecewise continuous and bounded function, which is periodic with period $T_2>0$. In this case, we identify $x(\cdot),~ \dot x(\cdot)$ with their periodic extensions from $[0,T]$ to $[0,\infty)$.
We then have the main theorem:
\begin{theorem}[Main Theorem] \label{th:main}
Consider  Cauchy problem \eqref{sys1*}, with $f=f_{T_2}$ as described above, and denote  $x \in C^1([0,\infty))$ given by \eqref{varpar} as its unique solution. Let $\alpha = \dfrac{T_2}{T_1}\in \mathbb R_+$, and if $\alpha \in \mathbb{Q}_+$, then we take $\alpha =m/n$ with $m,n \in \mathbb N$ in lowest terms. In that case, denote $T_3 = nT_2 = m T_1$. 

\begin{itemize}
  \item[1.] If $ \alpha \notin \mathbb{Q}_+$, then $x(\cdot) \in L^{\infty}(\mathbb R_+)$ and not periodic for any period $T \in \mathbb R_+$.
  \item[2.] If $\alpha =\frac{m}{n} \in \mathbb{Q}_+ \setminus \mathbb{N},$ then $x(\cdot)\in L^{\infty}(\mathbb R_+)$ and periodic with  $T_3. $
  \item[3.] If $\alpha = \frac{m}{n} \in \mathbb{N},$ and $$\Big(f(t),\cos(\omega_0 t)\Big)_{L^2(0,T_3)} = \Big(f(t),\sin(\omega_0 t)\Big)_{L^2(0,T_3)} = 0,$$ then $x(\cdot) \in L^{\infty}(\mathbb R_+)$ is periodic with period  $T_3$.
  \item[4.] If $\alpha = \frac{m}{n} \in \mathbb{N},$ and  $$\Big(f(t),\cos(\omega_0 t)\Big)_{L^2(0,T_3)}\neq 0~\text{ or }~\Big(f(t),\sin(\omega_0 t)\Big)_{L^2(0,T_3)}\neq 0,$$ then $x(\cdot)$ is resonant. In particular, it is not periodic of any period $T\in\mathbb R_+$ and $x \notin L^{\infty}(\mathbb R_+)$.
\end{itemize} 
\end{theorem}
The proof of Theorem \ref{th:main} constitutes Section \ref{mainproof}. 

\subsection{Discussion of Main Result and Context}
Having recalled the basic properties of the oscillator with sinusoidal forcing, as well as resonance, we can   discuss the result in Theorem \ref{th:main} in that context. First, we define {\em resonance}, corresponding to solutions to \eqref{oscillator}. 
\begin{definition} Given a forcing function $f=f_T$ which is both periodic (with minimal period $T>0$) and bounded (in time), the resulting solution $x(\cdot)$ to \eqref{oscillator} is {\em resonant} if it is unbounded. \end{definition}
\noindent In this sense, resonance constitutes an unbounded response  to a bounded input. 

The variation of parameters formula, given above in \eqref{varpar}, involves a term of form $\ds \int_0^tf_T(\tau)\cos(\omega_0 \tau) d\tau$, (of course also a term with sine). This term can be evaluated for $t=T$ (on periods of $f$) or periods of $\frac{2\pi}{\omega_0}$. 
In this scenario, we can obtain a simple illustration of resonant behavior. Let us assume that we have a function $f$ and a value $T_3>0$ so that:
\begin{itemize}
\item $f_T$ and $\cos(\omega_0t)$ have a coincident period $T_3$ (though $T_3$ need not be equal to $\frac{2\pi}{\omega_0}$, for instance if $T_3$ is an integer multiplier thereof), and assume (for simplicity) that $(f,\sin(\omega_0 t))_{L^2(0,T_3)}=0$;
\item $\ds \int_0^{T_3}f(\tau)\cos(\omega_0\tau)d\tau =Q \neq 0$.
\end{itemize} 
Then, for $t=nT_3+r$ with $n\in \mathbb N$ maximal and $0 \le r < T$, we have: 
\begin{align*}\int_0^tf(\tau)\cos(\omega_0\tau)d\tau = & ~ \int_0^{nT_3}f(\tau)\cos(\omega_0\tau)d\tau + \int_0^rf(\tau)\cos(\omega_0\tau)d\tau\\
=&~ n \int_0^{T_3}f(\tau)\cos(\omega_0\tau)d\tau+\int_0^rf(\tau)\cos(\omega_0\tau)d\tau \\
=&~nQ+\int_0^rf(\tau)\cos(\omega_0\tau)d\tau.
\end{align*}
From here, we observe several items:
\begin{itemize}
\item Resonance is permitted by the existence of a coincident period for $f$ and $\cos(\omega_0t)$.
\item The growth per cycle in the resonant case (which we denoted $Q$, related to the previously defined amplification factor $\mathbf A$) is obtained by an $L^2(0,T_3)$ inner product of the form $Q=\ds(f,\cos(\omega_0 t))_{L^2(0,T_3)}$, and its non-zero value.
\item Immediately, we see  the fundamental set $\{\cos(\omega_0t),\sin(\omega_0t)\}$ playing a special role in this general setting, in particular, through the projection of $f$ onto the span of this set in the sense of $L^2(0,T_3)$.
\end{itemize}

In our definition of {\em resonance},  we do not include any mention of a quasi-period or lack of a period. Though in  resonant cases, we apparently have a quasi-period, which can be observed, along with an explicit growth envelope, which can be calculated. It is interesting to note that, in the classical case where $f$ is a sinusoid, denoted $f(t)=\text{trig}(\omega t)$, the orthogonality of sinusoids is in force. This forces a ``0 or not" dichotomy for the aforementioned inner product against $\text{trig}(\omega_0 t)$, which {\em may not be the case with a general $f_T$}. Indeed, a general $f_T$ may aggregate mass or produce cancellation with $\text{trig}(\omega_0t)$ on some interval $[0,T_3]$. At the heart of our considerations is whether $f_T$ ``picks out" some non-zero mass from the underlying fundamental set on the period $T_3$. As one might suspect, the symmetries of the graph of $f_T$ are fundamentally at issue here---both in terms of output values relative to the semi-period $t=T/2$ as well as output values relative to half of the maximum amplitude. Such considerations will be exemplified in our examples and are accounted for in Theorem \ref{th:main} above.

\section{Motivational Examples with Non-sinusoidal Forcing}
In this section we discuss forcing functions of interest, and provide some examples of solutions with non-sinusoidal forcing functions $f_T$. 
\subsection{Forcing Functions of Interest}
We now describe some interesting forcing functions $f_T$ which are viable in the context of Theorem \ref{th:main}. Some of these functions will be considered in simulated solutions to \eqref{oscillator}. We have a particular interest in non-smooth functions, noting that switching functions arise naturally in engineering and mechanics applications. We will include a brief description of why these functions are mathematically interesting, from the point of view of periodicity and resonance. For each example, we plot  the function  on a representative interval $[0,1]$, taken to be one full period $T=1$ of the graph. 
\begin{itemize}
\item So-called {\em step functions}, which may or may not be symmetric about a semi-period ($T/2$). \\ We note that the output values of $f_T$ are not relevant, since we can translate vertically by a constant factor, yielding (by superposition) only a fixed translation of the solution. 
\begin{center}
\includegraphics[scale=.37]{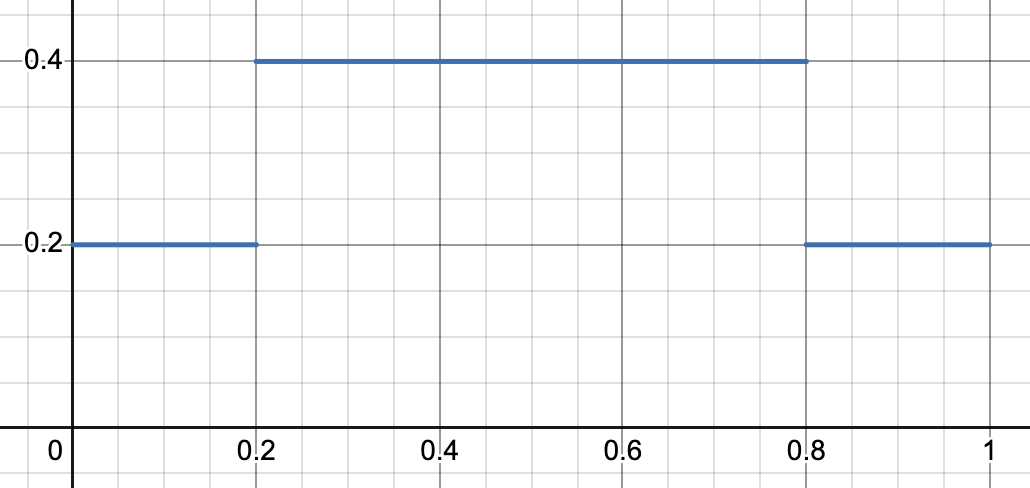} \hskip.5cm \includegraphics[scale=.35]{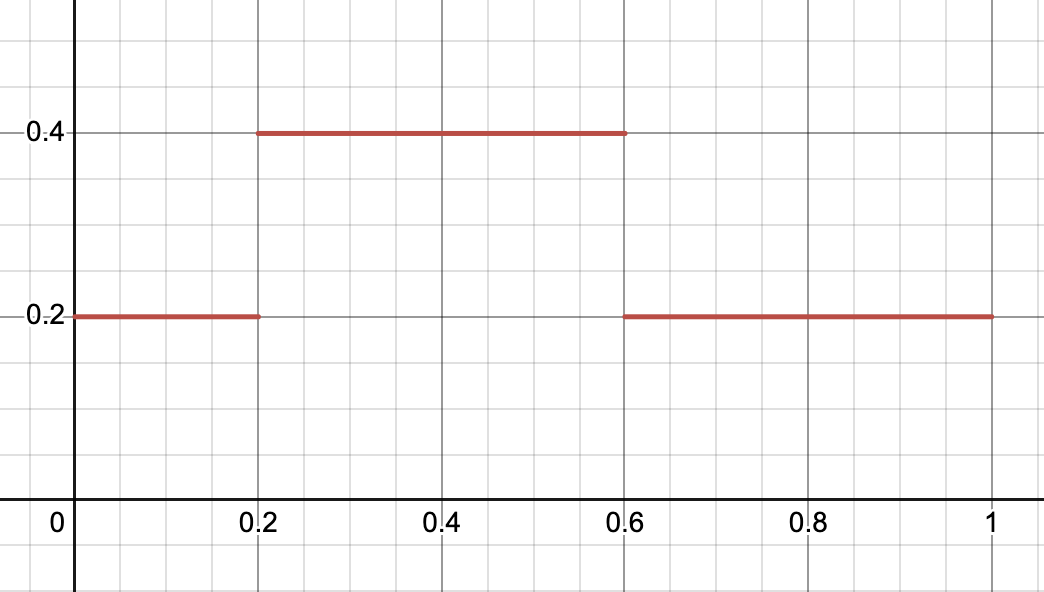}
\hskip.3cm
\end{center}
\begin{quote} \small The graph (L) is symmetric about the semi-period. The graph (R) is periodic on $[0,T]$, but not symmetric about the semi-period. \end{quote}

\item {\em Rectified trigonometric functions}, which compose absolute value (L) or max (R) with  $\text{trig}(\omega t)$:
\begin{center}
\includegraphics[scale=.31]{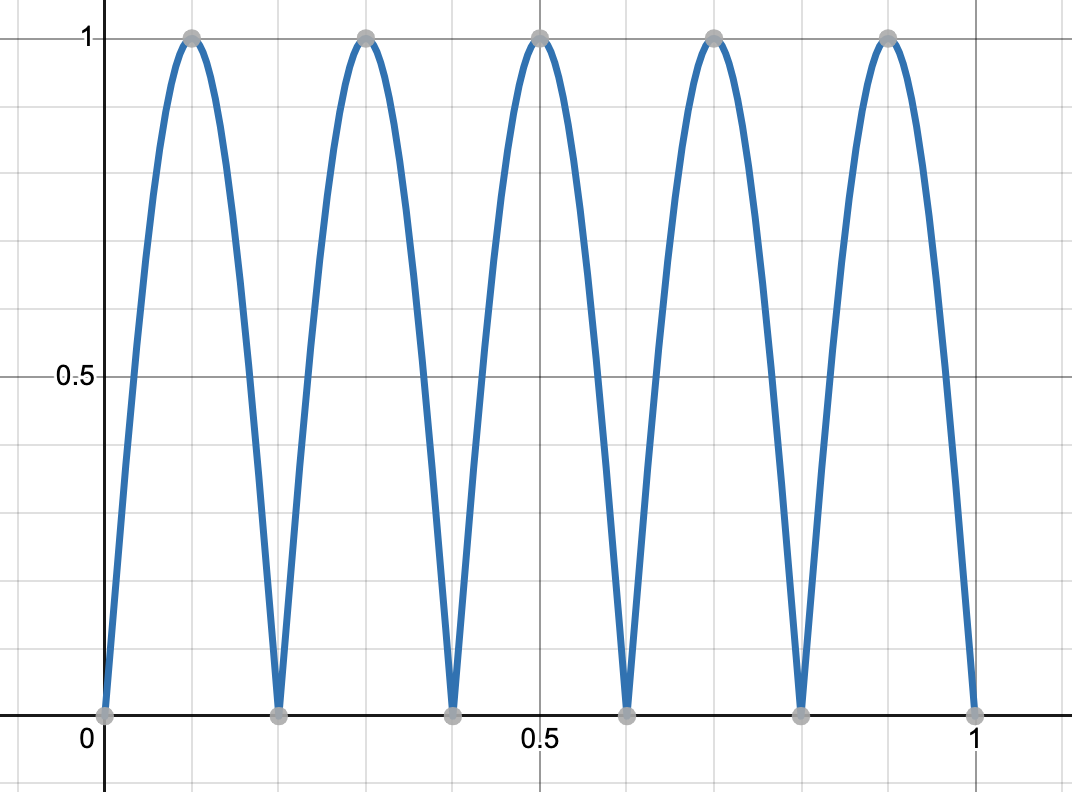} \hskip.5cm \includegraphics[scale=.38]{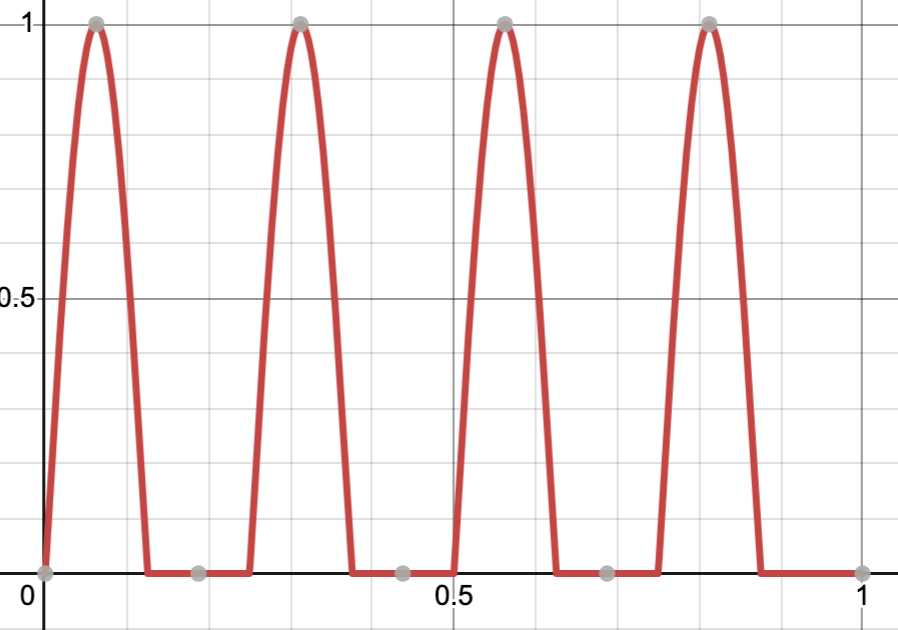}
\end{center}
Such functions are not symmetric in their output values (relative to, say, their average value), and present no cancellation on any fractional period, regardless of translation. These functions are continuous, but not continuously differentiable. 

\item {\em Piecewise linear functions} (continuous, but not differentiable):
\begin{center}
\includegraphics[scale=.55]{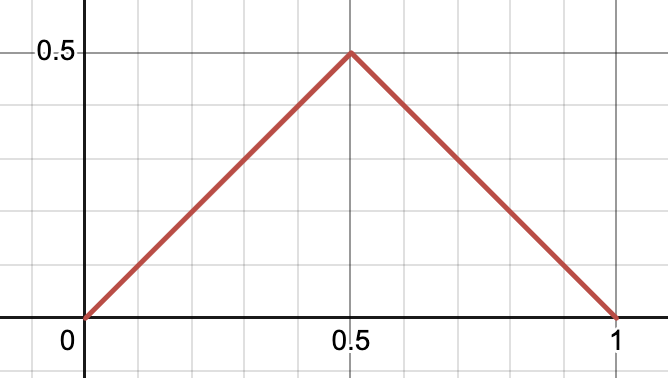}
\hskip.5cm
\includegraphics[scale=.3]{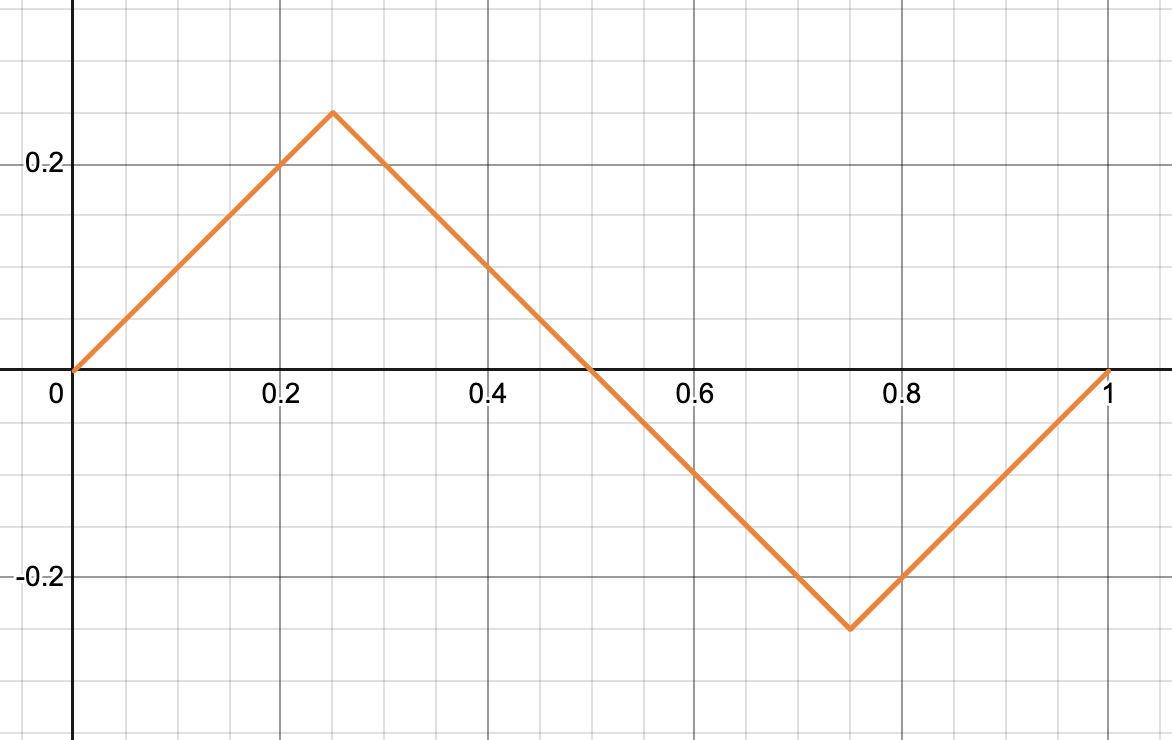}
\end{center}
\begin{quote} \small In the above examples, the forcing function  represents a  {\em sawtooth} function. The (L) is symmetric about the semi-period and positive.  The graph (R) a is a piecewise linear function, producing a sawtooth which mirrors a sinusoid, to some extent. This produces a similar ``cancellation" on its period, resulting in symmetry about the output value $0$, and anti-symmetry about the semi-period. \end{quote}

\end{itemize}
\subsection{Examples of Responses}
In this section we look at ODE solutions to \eqref{sys1*}---the oscillator with periodic forcing and zero Cauchy data. We will utilize the functions in the previous section which are piecewise continuous and bounded, but not necessarily smooth, and not sinusoidal. 
We look for several examples (of  forcing and solution) which illustrate the cases of Theorem \ref{th:main}. We are interested in examples which demonstrate the existence of a:
\begin{itemize}
\item  non-sinusoidal forcing function that produces resonance with period matching ($T_1=T_2$);
\item  non-sinusoidal forcing function that produces resonance without period matching, but a particular multiple occurs, such as $k T_1=T_2$ for $k \in \mathbb N$;
\item  non-sinusoidal forcing yielding a periodic (non-resonant) solution with period matching;
\item non-sinusoidal forcing yielding a bounded but non-periodic solution.
\end{itemize}

These examples highlight the main outcomes related to bullet points (3.) and (4.) in Theorem \ref{th:main}, and these may be somewhat unexpected scenarios; from Section \ref{classical} for sinusoidal forcing, resonance occurs if and only if period matching is in place. For our result on general forcing, (a) {\em resonance can occur if the forced period is an integer multiple of the natural period}, and (b) {\em period-matching is not sufficient to ensure resonance}. 

\subsubsection{Examples: Periodic Solutions and Resonance with Sawtooth Forcing}
We consider a sawtooth forcing  that mirrors a sinusoid, as in the previous section. Denote $f_{st}(t)$ as the periodic extension to $ t\in [0,\infty)$ of the graph below.
\begin{center}
\includegraphics[scale=.32]{saw3.png}
\end{center}
The forcing function is graphed with a period of $T_2=1$, but we will consider several different periods. For now, recall that our ODE is:
$$\ddot x + \omega_0^2x = f_{st},~~x(0)=\dot x(0)=0.$$
Below, we will construct several examples of periodicity and resonance using sawtooth functions; these examples are interesting in  that they are not sinusoids, but  mirror them in their qualitative structure.

First, let us give $f_{st}(t)$ sawtooth forcing with a $T_2=4$, and we consider $T_1=4$ (with $\omega_0=\pi/2$). In this instance, we have {\em period matching}.  While the forcing is not a sinsuoid, it is clear that there is nontrivial ``agreement" between ~$f_{st}$ here and $f(t)=\sin(\frac{\pi}{2} t)$---the latter of which causes traditional resonance as a forcing. The solution with $f=f_{st}$ is shown below.
\begin{center} \includegraphics[scale=0.45]{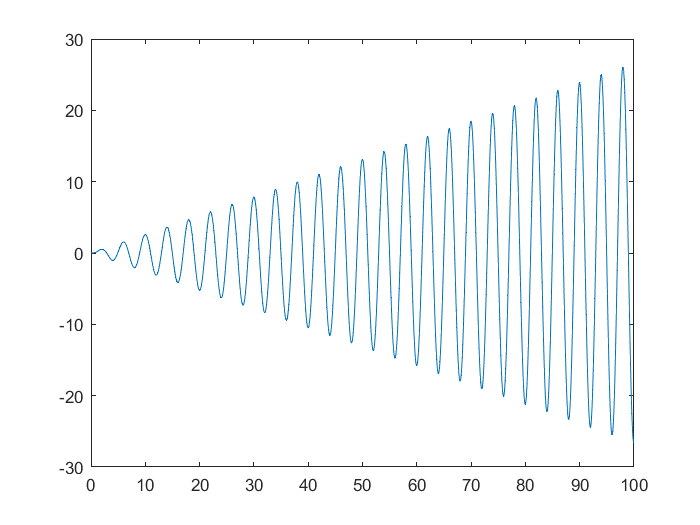} \end{center}
As the graph indicates, resonance occurs; this validates case (4.) of Theorem \ref{th:main}, as period matching is in place and 
$$\int_0^4f_{st}(t)\sin([\pi/2]t)dt = \dfrac{16}{\pi^2} \neq 0.$$

We can subsequently consider a scenario where period matching is not in force. Consider  $f_{st}(t)$ with $T_2=2\pi$, and let us consider the ODE with $T_1=4\pi$ (so $\omega_0=1/2$). In this case $T_2/T_1=1/2 \in \mathbb Q_+\setminus \mathbb N$. We expect boundedness of the solution, with period $T_3=2\pi$ from case (2.) of Theorem \ref{th:main}. The solution is shown below.
\begin{center} \includegraphics[scale=0.45]{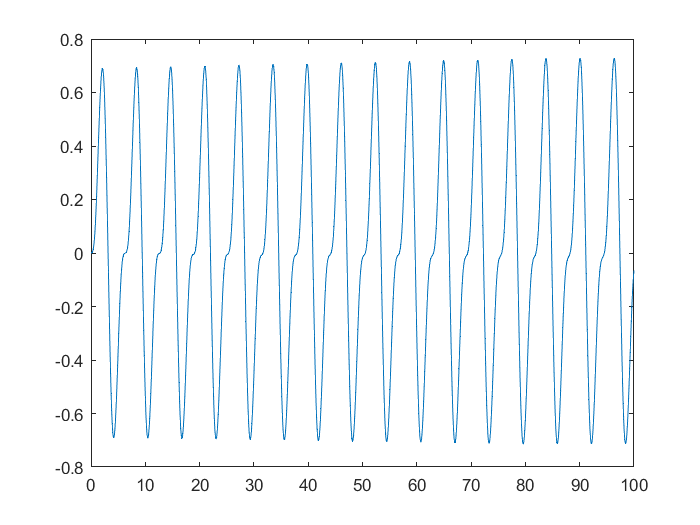} \end{center}

We can finally try $f_{st}$ with a period of $T_2 = 6$, with $T_1=2$. This presents an integer multiple of the natural period, and we do observe resonance. 
\begin{center} \includegraphics[scale=0.35]{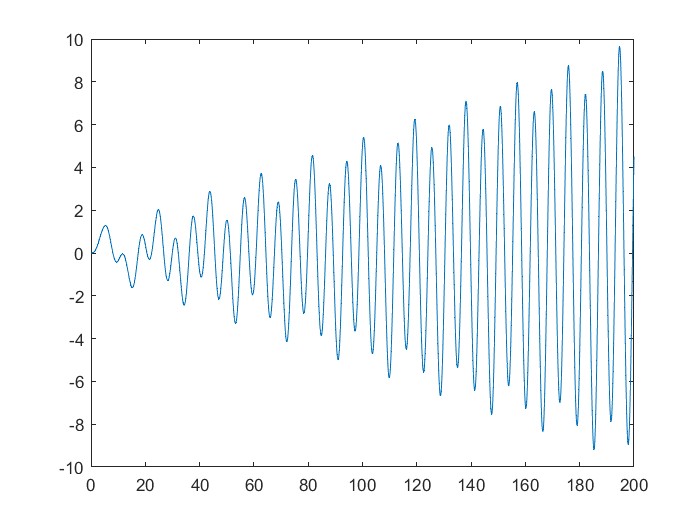} \end{center}
This scenario is interesting, because if we were to utilize a sinusoidal forcing which ``matches" this sawtooth function, namely ~$f(t)=\sin([\pi/3]t)$, we would not observe resonance, since $T_1=2 \neq T_2=6$. 
Checking this condition against point (4.) in Theorem \ref{th:main}, we note that
$$\int_0^6f_{st}(t)\sin(\pi t)dt \approx -.27019 \neq 0,$$ yielding resonance, while
$$\int_0^6 \sin([\pi/3]t)\sin(\pi t) dt = 0,$$
by classical orthogonality of sinusoids.

\subsubsection{Example: Resonance with A Step Function and Period Matching}
We now demonstrate resonance with a properly discontinuous forcing. Consider a piecewise constant step function which is symmetric about its half period and takes on both step values on sets of equal measure.
We consider $T_2=2$ in this instance, and we consider the periodic extension to $[0,\infty)$ of 
\[f_{s}(t) =  \begin{cases} 
      1 & t \in [0,T_2/4],\\
      0 & t \in [T_2/4,3T_2/4],\\
      1 & t \in [3T_2/4,T_2].
   \end{cases}
\]
We consider the oscillator, taking $\omega_0=\pi$ with $f=f_s$ above, so $T_1=T_2=2$, and we observe:
\begin{center} \includegraphics[scale=0.4]{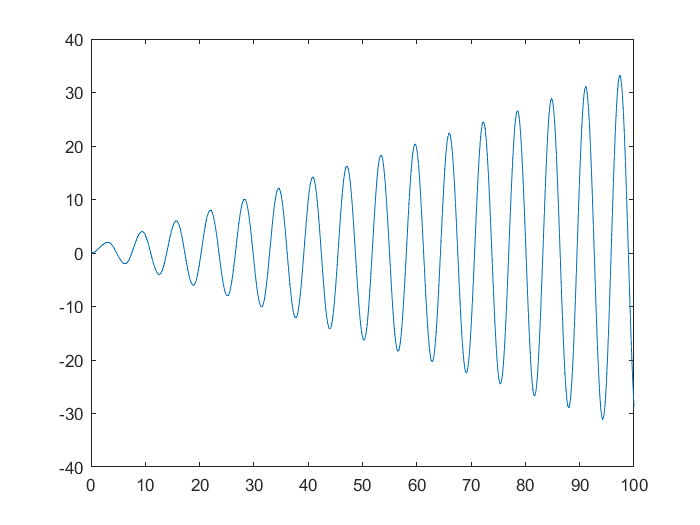} \end{center}
Again, resonance is observed in this scenario of period matching. 

\subsubsection{Example: Resonance with Positive Function and No Period Matching} 
We now consider a  non-negative function that can produce resonance. We consider the oscillator with $\omega_0=\pi$ and we consider a forcing of the form  $f_{a}(t) = \max\left\{\sin(\frac{\pi}{2} x),0\right\}.$ In this case, $T_2=4,$ while $T_1=2$, so $\alpha = T_2/T_1=2 \in \mathbb N$ (in case (3.) or (4.) of Theorem \ref{th:main}). We  observe resonance.
\begin{center} \includegraphics[scale=0.4]{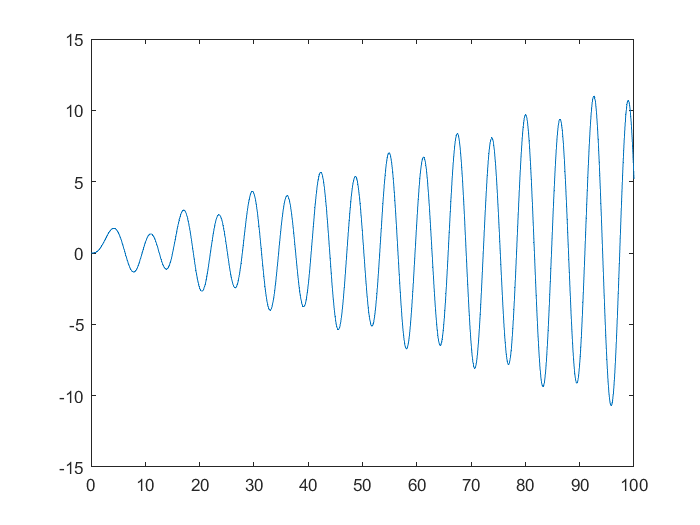} \end{center}
A posteriori, based on the dichotomy for (3.) and (4.) in Theorem \ref{th:main}, we infer that
$$\Big(f_a(t),\cos(\pi t)\Big)_{L^2(0,4)}\neq 0~\text{ or }~\Big(f_a(t),\sin(\pi t)\Big)_{L^2(0,4)}\neq 0,$$

\subsubsection{Example: Period Matching with No Resonance}
We now consider a situation where we have {\em explicit period matching} ($T_1=T_2$), but no resonance. In this case we are tacitly enforcing the condition in bullet point (3.) of Theorem \ref{th:main}, namely, a forcing that has integer period matching (in fact, exact period matching) but no proper first (Fourier) mode contribution. So:
$$\Big(f(t),\cos(\omega_0 t)\Big)_{L^2(0,T)} = \Big(f(t),\sin(\omega_0 t)\Big)_{L^2(0,T)} = 0.$$
This is to say: cancellations which occur between $f$ and the fundamental set, on each period, ensure that there is zero growth factor $Q$ per period. So, consider the step function below:
\begin{center} \includegraphics[scale=0.42]{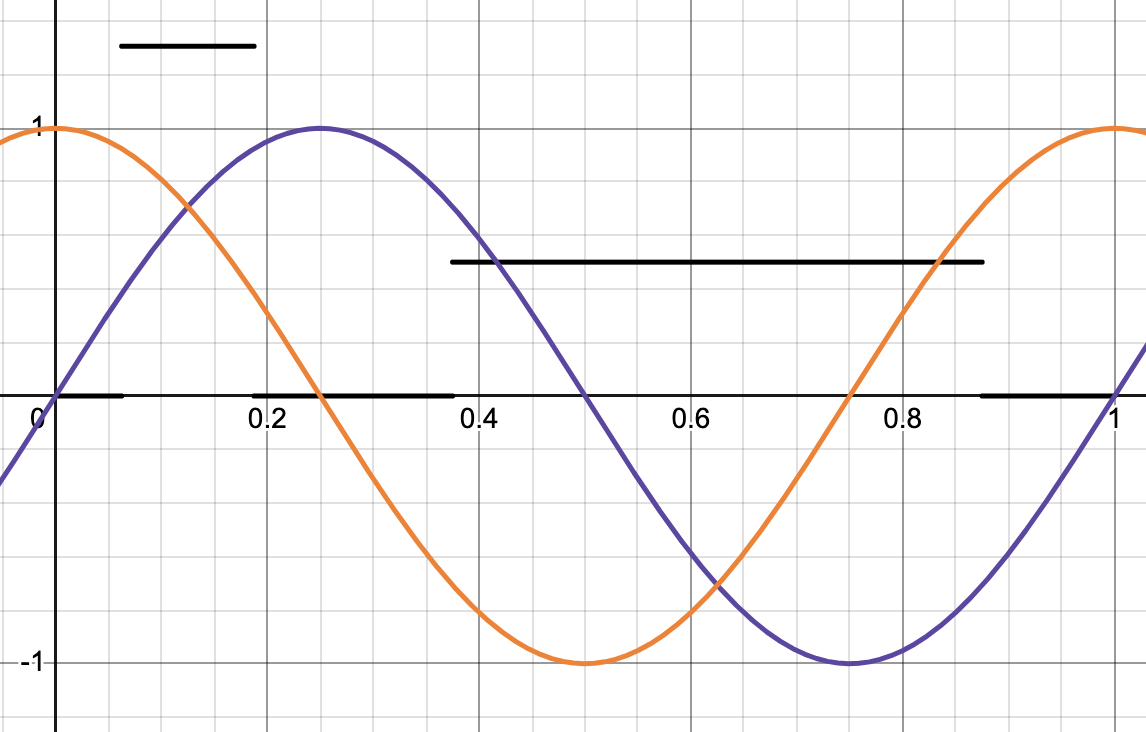} \end{center}
This function, $f_m(t)$ graphed in black, is defined on $t \in [0,1]$, then extended periodically to $t \in [0,\infty)$. It assume two positive values, say $A<B$, which have a particular relationship, described below, inducing cancellation. We include the graphs of ~$\sin(2\pi t)$ and ~$\cos(2\pi t)$ for reference. Then, denoting $T=1$:
\[f_{m}(t) =  \begin{cases} 
      B & t \in [1/16,3/16],\\
      A & t \in [3/8,7/8],\\
      0 & t \in [0,1/16] \cup [3/16, 3/8] \cup [7/8,1].
   \end{cases}
\]
If we let, for instance, $A=1/2$, then the value of $B$ can be chosen as
$$B=-\dfrac{\int_{3/8}^{7/8}\cos\left(2\pi\cdot x\right)dx}{2\int_{1/16}^{3/16}\cos\left(2\pi\cdot x\right)dx}.$$
With these choices, it is readily verified that 
$$\Big(f_m(t),\cos(2\pi t)\Big)_{L^2(0,1)} = \Big(f_m(t),\sin(2\pi t)\Big)_{L^2(0,1)} = 0,$$
and hence we have period matching for this forcing, but resonance will not occur and we will obtain a periodic solution of period $T_3=1=T_1=T_2$.

\subsubsection{Example: Irrational Relationship Between $T_1$ and $T_2$}
When the relation $\alpha = T_2/T_1 \in \mathbb Q_+\setminus \mathbb N$, the solution mirrors exactly the case of classical sinusoidal forcing, described above; namely, if there is an irrational relationship between periods (the periods are not commensurate), we obtain a bounded solution that is not periodic of any period. Consider  the forcing is $f_i(t) = \sin(\pi t)$, which has period $T_2 = 2$. We then consider $\omega_0=1$ in the oscillator ODE, so that $T_1 = 2\pi$. The corresponding solution forced by $f_i(t)$ is below.
\begin{center} \includegraphics[scale=0.35]{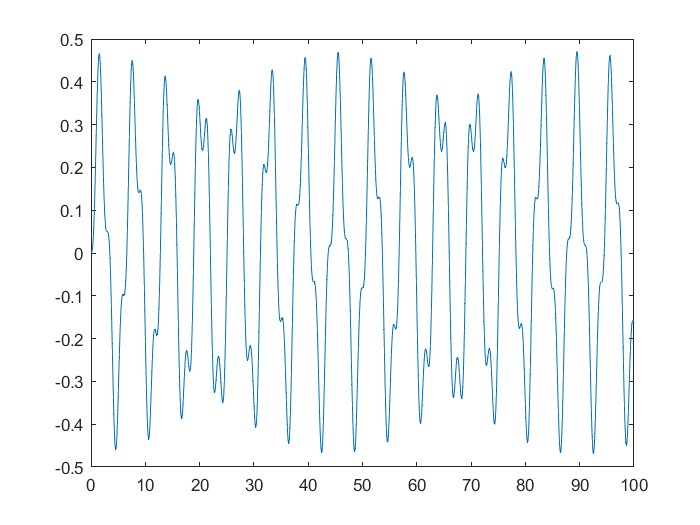} \end{center}

We can also consider, for a non-sinusoidal example, the rectified sine function $f_{j}(t) = |\sin(2\pi t)|$, which has period $T_2=1/2$. 
Again, considering the natural period of $T_1=2\pi$ (when $\omega_0=1$), we observe the solution in blue (plotted against the forcing function in green). 
\begin{center} \includegraphics[scale=0.28]{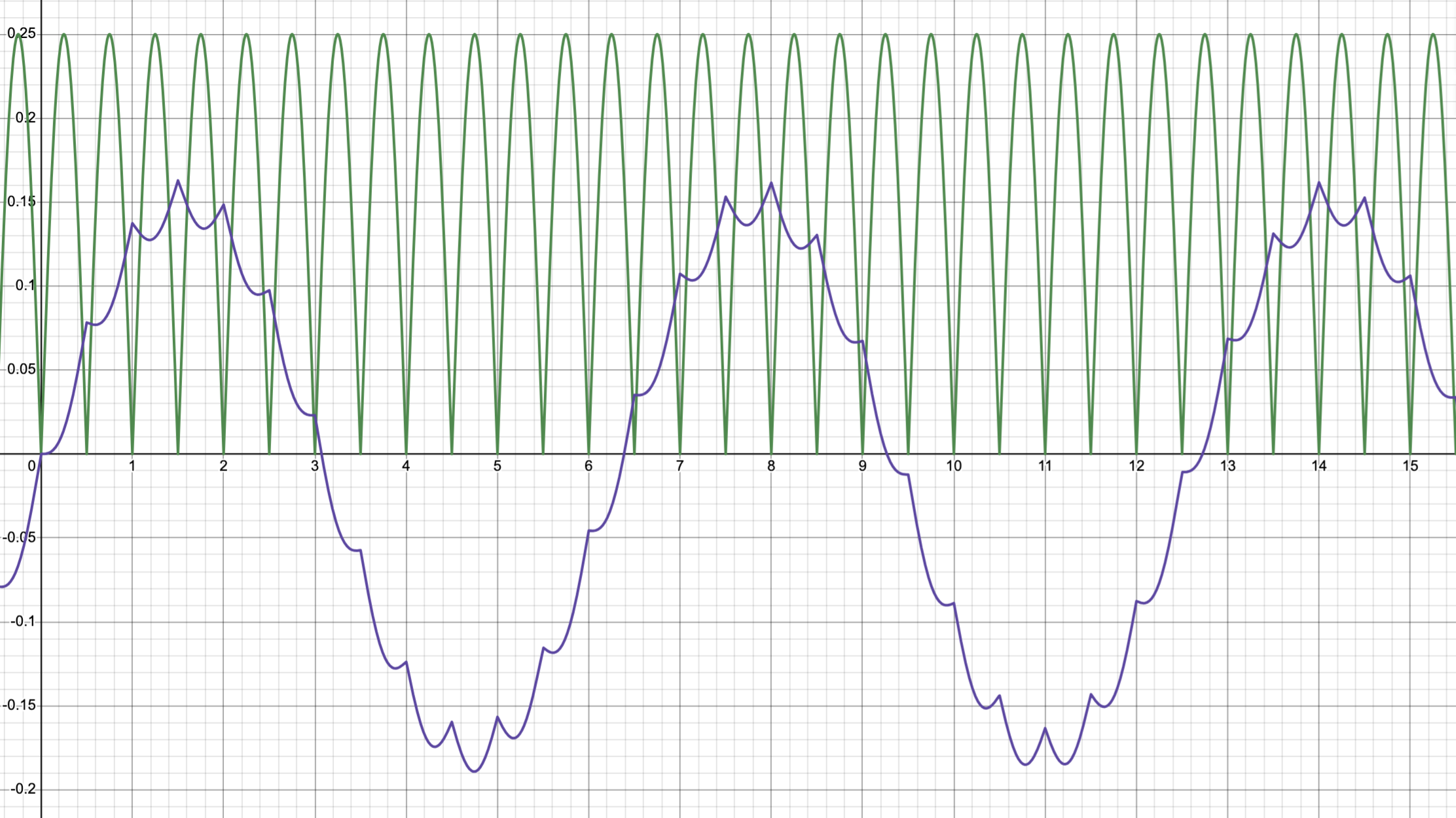} \end{center}
Again, we note that this function is bounded, but  not periodic of any period $T \in \mathbb R_+$ (one can see this by focusing on local max and mins of the carrier sinusoid).

\section{Proof of Main Theorem}\label{mainproof}
We retain our standing notation, and recall it briefly for the system
\begin{equation}\label{sys1**}\begin{cases} 
      \ddot x(t) + \omega_0^2 x(t) = f(t), \\
      x(0) = 0; ~~\dot x(0) = 0.
   \end{cases}
\end{equation}
The natural frequency is $\omega_0 \in \mathbb{R_+}$ and the associated natural period of the system is $T_1 = \frac{2\pi}{\omega_0}>0$. We assume the forcing $f:[0,\infty)\rightarrow\mathbb{R}$ is a piecewise continuous and bounded function, which is periodic with (minimal) period $T_2$. The quantity $\alpha = T_2/T_1$; when $\alpha = m/n \in \mathbb Q_+$ we assume it is in lowest terms, and we denote by $T_3=nT_2=mT_1$.

We proceed with some elementary facts and then proceed to proving the main theorem. 

\subsection{Useful Facts, and Identities}
We first recall Theorem \ref{sumprod}, which states that for continuous periodic functions, sums (and products) thereof are periodic if and only if the periods of the constituent functions are commensurate (rational multiples of one another). 
We can now turn to the calculus of periodic functions. 
\begin{lemma} Let $f: [0,\infty)\to \mathbb R$ be a continuous periodic function of (minimal) period $T>0$. Then for any derivative of $f$ which exists, said derivative will be periodic with period $T$. In particular, if $f \in C^1([0,\infty); \mathbb R)$ and is periodic of period $T$, then $f'(\cdot)$ is periodic with period at most $T$. \end{lemma}
\begin{proof}
For all $t \in [0,T)$ we have $f(t)=f(t+T)$. Differentiating both sides yields, for all $t \in [0,T)$, that $f'(t)=f'(t+T)$. 
\end{proof}
For antiderivatives, it is natural to require zero mean. 
\begin{lemma} \label{difflemma} Suppose that
$g:[0,\infty) \to \mathbb R$ is an integrable function (for instance, $g$ is continuous). If $g$ is periodic with period $T$ and 
~$\int_0^Tg(t)dt =0,$ then any antiderivative $\int^tg(\tau)d\tau$ is also periodic, with period at most $T$. 
\end{lemma}
\begin{proof} Let us extend $g$ periodically from $[-T,0]$ and denote the extension again by $g$. Since $g$ has zero average on $[0,T]$, and we have extended it periodically on $[-T,0]$, we know that $\int_{-T}^0 g(\tau)d\tau=0$. 

WLOG, we choose the antiderivative  $G(t)= \int_0^t g(\tau)d\tau$.
Write the following for $G(t)$:
$$G(t+T)=\int_0^{t+T}g(\tau)d\tau=\int_{-T}^{t}g(\tau)d\tau = \left(\int_{-T}^{0}+\int_{0}^t\right)g(\tau)d\tau=0+G(t).$$ This holds for all $t \in [0,T]$, so $G(\cdot)$ is periodic with period at most $T$.
\end{proof}

We now note that any period for a given periodic function must be a multiple of the minimal period. This provides some motivation for including minimality in the definition. 
\begin{lemma} \label{lemma3} Suppose that $g:\mathbb [0,\infty) \to \mathbb R$ is periodic with period $T_0$. Suppose for $T_1 \in \mathbb R_+$ that $g(t+T_1)=g(t)$ for all $t \in [0,\infty)$. Then there is an $n \in \mathbb{N}$ such that $T_1 = nT_0.$ 
\end{lemma}

\begin{proof}
    Let $T_1$ be a period  of $g$ and, for assume for contradiction  $\forall n \in \mathbb{N}, T_1 \neq nT_0$. Therefore, by the division algorithm, we take $T_1 = n_1 T_0 + r$ where $n_1 \in \mathbb{N}$ and $0 < r < T_0$. Since $g(t)$ is periodic with period $T_0$, we have $g(t) = g(t+T_0) = ... = g(t+n_1 T_0)$ for all $t \in [0,\infty)$. We can translate this  equation by $r>0$ to obtain $g(t+r) = g(t+n_1 T_0 +r) = g(t+T_1).$ Therefore, since $g(t)$ is also periodic with period $T_1, $ we know that $g(t) = g(t+r),$ and so $g$ is also periodic with period $r$. However, since we assumed that $T_0$ is the minimum period, but observed that $0<r < T_0$ is a period, we have a contradiction.
\end{proof}

The following lemmas help achieve boundedness in the proofs below supporting Theorem \ref{th:main}. 
\begin{lemma}\label{lemma1} Suppose $n \in   \mathbb{N}$ and $w \in   \mathbb{R}$ with $\frac{w}{2\pi} \notin \mathbb Z$. 
  Then \begin{equation} \sum_{k=0}^n \sin(kw) = \frac{\sin\left(\frac{wn}{2}\right)\sin\left(\frac{w(n+1)}{2}\right)}{\sin(w/2)}.\end{equation}
\end{lemma}
\begin{proof}
The sine summation can be rewritten using Euler's identity, using the imaginary part $\mathscr I(\cdot)$, as $$\sum_{k=0}^{n} \sin(wk) = \mathscr{I}\left(\sum_{k=0}^{n} \cos(wk)+ i\sin(wk)\right) = \mathscr I\left(\sum_{k=0}^{n} (e^{iw})^k\right).$$ By hypothesis, $w$ is not an integer multiple of $2\pi$, and thus $e^{iw} \neq 1$. Therefore, the finite geometric series is equal to 
\begin{equation}\sum_{k=0}^{n} e^{iwk}=\dfrac{e^{iw(n+1)}-1}{e^{iw}-1} = \dfrac{e^{iw(n+1)/2}\left[e^{iw(n+1)/2}-e^{-iw(n+1)/2}\right]}{e^{iw/2}\left[e^{iw/2}-e^{-iw/2}\right]}.\end{equation} 
We observe \begin{align*} e^{iw(n+1)/2}-e^{-iw(n+1)/2} =&~ i\sin(w(n+1)/2) + \cos(w(n+1)/2) - i\sin(-w(n+1)/2) - \cos(-w(n+1)/2) \\ =& ~2i\sin(w(n+1)/2).
\end{align*}
Thus we have
\begin{align*}
\mathscr I\left(\sum_{k=0}^{n} e^{iwk}\right) = & ~\mathscr I\left(\dfrac{e^{iw(n+1)/2}(2i\sin(w(n+1)/2))}{e^{iw/2}(2i\sin(w/2))}\right)  \\
=&~\mathscr I\left(\dfrac{e^{iwn/2}(\sin(w(n+1)/2))}{\sin(w/2)}\right).
\end{align*}
The denominator above is defined, since $w/2 \notin \mathbb Z$, as $w \notin 2\pi\mathbb Z$ by assumption. 
Using Euler's formula once more for the exponential, we have
\begin{equation} \mathscr I\left(\sum_{k=0}^{n} e^{iwk}\right) =\mathscr I\left([i\sin(wn/2) + \cos(wn/2)]\left(\dfrac{\sin(w(n+1)/2)}{\sin(w/2)}\right)\right).\end{equation}
Finally, we obtain
\begin{equation}\sum_{k=0}^{n} \sin(wk) = \mathscr I\left(\sum_{k=0}^{n} e^{iwk}\right)=\dfrac{\sin(wn/2)\sin(w(n+1)/2)}{\sin(w/2)},\end{equation}
as desired.
\end{proof}
The lemma for cosine follows similarly, using the real part in the above argument.

\begin{lemma}\label{lemma2} Suppose $n \in   \mathbb{N}$ and $w \in   \mathbb{R}$ with $\frac{w}{2\pi} \notin \mathbb Z$. Then \begin{equation} 
\sum_{k=0}^n \cos(kw) = \dfrac{\cos\left(\dfrac{wn}{2}\right)\sin\left(\dfrac{w(n+1)}{2}\right)}{\sin(w/2)}.\end{equation}
\end{lemma}
%
%
%
%
%
%
%
%
%

\subsection{Proof of (1.) in Theorem \ref{th:main}}
\begin{proof} Assume that $T_2/T_1=\alpha \notin  \mathbb{Q}_+$. We will first show that the solution $x(t)$ is bounded. 

By the variation of parameters formula and the initial conditions $x(0) = \dot x(0) = 0$, the unique solution of the oscillator ODE is
\begin{equation} \label{integrandy} x(t) = \frac{\sin(\omega_0 t)}{\omega_0} \int_0^t f(\tau)\cos(\omega_0 \tau)d\tau - \frac{\cos(\omega_0 t)}{\omega_0} \int_0^t f(\tau)\sin(\omega_0 \tau)d\tau.\end{equation}
As $f \in L^{\infty}(\mathbb R_+)$, the integrands in \eqref{integrandy} are bounded. Multiplying by sinusoids, of course preserves boundedness, so we must only show that  the time integrals are bounded for all $t \in [0,\infty)$. This will be accomplished via Lemmas \ref{lemma1} and \ref{lemma2}. 

We will first observe that for all $n_1 \in  \mathbb{N}$,
\begin{equation}\label{summ}
\sum_{k=0}^{n_1} \cos(2\pi k \alpha) \le \dfrac{1}{|\sin(\pi \alpha )|}.\end{equation}
To that end,  let $w = 2\pi\alpha$. Therefore, since $\pi \alpha \notin \pi\mathbb N$, Lemma \ref{lemma2} holds in this case. The cosine sum is thus represented as
\begin{equation}\sum_{k=0}^{n_1}\cos(wk) = \frac{\cos(wn_1/2)\sin(w(n_1+1)/2)}{\sin(w/2)}.\end{equation}
Taking absolute values and replacing the numerator $|\cos(wn_1/2)\sin(w(n_1+1)/2)|\le1$, we obtain, with
$w=2\pi \alpha$, the inequality \eqref{summ}. Similarly, invoking Lemma \ref{lemma1}, yields \begin{equation}\sum_{k=0}^{n_1}\sin(wk) = \frac{\sin(wn_1/2)\sin(w(n_1+1)/2)}{\sin(w/2)}.\end{equation}
From this, we deduce that
 \begin{equation}\label{summ1} \sum_{k=0}^{n_1} \sin(2\pi k \alpha) \le \dfrac{1}{|\sin(\pi \alpha)|}.\end{equation}
 Since $\alpha \notin \mathbb Q_+$, we observe that the denominator of the RHS does not vanish, and the summation is finite. 

We now proceed to bound $$\int_0^t f(\tau)\sin(\omega_0 \tau)d\tau~~\text{ and }~~\int_0^t f(\tau)\cos(\omega_0 \tau)d\tau$$ for all $t \in [0,\infty)$.
Let $t \in [0,\infty)$. By the division algorithm, with $T_2>0$ given, we know that $t = n_2 T_2 + r_2$ where $n_2 \in  \mathbb{N}$ and $0 \leq r_2 < T_2$. We then break the integral term as
$$\int_0^tf(\tau)\sin(\omega_0\tau)d\tau=\sum_{k=0}^{n_2-1}\int_{kT_2}^{(k+1)T_2} f(\tau)\sin(\omega_0 \tau)d\tau+ \int_{n_2 T_2}^{n_2 T_2 + r_2} f(\tau)\sin(\omega_0 \tau)d\tau.$$

By periodicity of the forcing function $f$, we know that $\forall t \in [0,\infty),~ f(t) = f(t+T_2) = ... = f(t+n_2 T_2).$ Thus, by direct substitution in the sum, the integral can be rewritten as
\begin{equation} \int_0^tf(\tau)\sin(\omega_0\tau)d\tau=\int_0^{T_2} f(\tau) \sum_{k=0}^{n_2-1} \sin(\omega_0 (\tau+kT_2)) d\tau + \int_{n_2 T_2}^{n_2 T_2 + r_2} f(\tau)\sin(\omega_0 \tau)d\tau.\end{equation}

We recall that $T_2 = \alpha T_1$ and $T_1 = \frac{2\pi}{\omega_0}$, thus  $\omega_0 kT_2  = 2\pi k \alpha$. Therefore by the sine addition formula, we have
\begin{equation*} \sum_{k=0}^{n_2-1} \sin(\omega_0 (\tau+kT_2)) =\sum_{k=0}^{n_2-1} \sin(\omega_0 \tau + 2\pi k\alpha) = \sin(\omega_0 \tau)\sum_{k=0}^{n_2-1} \cos(2\pi k \alpha) + \cos(\omega_0 \tau) \sum_{k=0}^{n_2-1} \sin(2\pi k \alpha). \end{equation*}

Using the triangle inequality and the bounds in \eqref{summ} and \eqref{summ1} we obtain
\begin{align}\left|\sum_{k=0}^{n_2-1} \sin(\omega_0 \tau + 2\pi k\alpha)\right| \le&~ |\sin(\omega_0 \tau)|\left|\sum_{k=0}^{n_2-1} \cos(2\pi k \alpha)\right| + |\cos(\omega_0 \tau)|\left| \sum_{k=0}^{n_2-1} \sin(2\pi k \alpha)\right|
\le&~\dfrac{2}{|\sin(\pi \alpha)|}. \end{align}

We now bound the integral expression in the variation of parameters formula:
\begin{align}
\left|\int_0^tf(\tau)\sin(\omega_0\tau)d\tau\right| \le &~\left|\int_0^{T_2} f(\tau) \sum_{k=0}^{n_2-1} \sin(\omega_0 \tau+k2\pi \alpha)) d\tau\right|+ \left|\int_{n_2 T_2}^{n_2 T_2 + r_2} f(\tau)\sin(\omega_0 \tau)d\tau\right| \\
\le & ~||f||_{L^{\infty}(0,T_2)}\int_0^{T_2}\left| \sum_{k=0}^{n_2-1} \sin(\omega_0 \tau+k2\pi \alpha)) d\tau\right|\\&+\nonumber ||f||_{L^{\infty}(0,T_2)}\int_{n_2 T_2}^{n_2 T_2 + r_2}|\sin(\omega_0 \tau)|d\tau \\ 
\le & ~||f||_{L^{\infty}(0,T_2)}\left[\dfrac{2T_2}{|\sin(\pi \alpha)|} + \int_{n_2 T_2}^{n_2 T_2 + r_2}|\sin(\omega_0 \tau)|d\tau\right] \\ 
\le & ~||f||_{L^{\infty}(0,T_2)}\left[\dfrac{2T_2}{|\sin(\pi \alpha)|} +r_2\right].
\end{align}
Hence we observe the boundedness 
\begin{equation}\label{thisone*}\left|\int_0^tf(\tau)\sin(\omega_0\tau)d\tau\right| \le||f||_{L^{\infty}(0,T_2)}\left[\dfrac{2T_2}{|\sin(\pi \alpha)|} +r_2\right].\end{equation}

In a completely analogous way,  $\ds \int_0^t f(\tau)\cos(\omega_0 \tau)d\tau$ is bounded, with identical bound as above in \eqref{thisone*}.
Boundedness of the solution $x(\cdot)$ follows immediately by applying the triangle inequality to the variation of parameters formula
\begin{equation*}x(t) = \frac{\sin(\omega_0 t)}{\omega_0} \int_0^t f(\tau)\cos(\omega_0 \tau)d\tau - \frac{\cos(\omega_0 t)}{\omega_0} \int_0^t f(\tau)\sin(\omega_0 \tau)d\tau\end{equation*}
and invoking the previous bounds.
In particular, we have
$$|x(t)| \le \dfrac{|\sin(\omega_0t)|}{\omega_0}\left|\int_0^tf(\tau)\cos(\omega_0\tau)d\tau\right|+\dfrac{|\cos(\omega_0t)|}{\omega_0}\left|\int_0^tf(\tau)\sin(\omega_0\tau)d\tau\right|,$$
from which we have the bound (taking the supremum for $t \in [0,\infty)$ and noting that $0\le r_2<T_2$)
\begin{equation}
||x(t)||_{L^{\infty}(\mathbb R_+)} \le \dfrac{2T_2}{\omega_0}||f||_{L^{\infty}(\mathbb R_+)}\left[\dfrac{2}{ |\sin(\pi \alpha)|}+1\right].
\end{equation}

We now proceed to show that the solution $x(t)$ is not periodic for any possible period $T \in \mathbb R_+$.
Assume for the sake of contradiction that $x(t)$ is periodic with some period $T \in (0,\infty).$ We will first constrain the possibilities for $T$ by invoking Lemma \ref{lemma3}. To do so, we note that on any open interval $I \subset [0,T_2]$ of continuity for $f$, the oscillator ODE in \eqref{sys1**} holds classically and $x(t) \in C^2(I)$. Let $A$ be the union of such intervals, i.e., $A=(0,T_2)\setminus \{x \in (0,T_2)~:~f \text{ is discontinuous at } x\}$. Then, since we have $x(t)=x(t+T)$ by hypothesis, by Lemma \ref{difflemma}, for all $t \in A$ we note that $\ddot x(t)=\ddot x(t+T)$. From this we infer from the ODE \eqref{oscillator}  that $f(t)=f(t+T)$ for all $t \in A$. However, for any $t_0 \notin A$, the value $t_0$ represents a point of either left or right continuity (by assumption on $f$). WLOG, say that $f$ is right continuous at $t_0$. Then for all $t \in A$ with $t>t_0$ we have $f(t+T)=\ddot x(t+T)+\omega_0^2x(t+T)=\ddot x(t)+\omega_0^2x(t)=f(t)$, and taking the limit as $t\searrow t_0$ yields that $f(t_0+T)=f(t_0)$. Hence, for all $t \in [0,T_2]$ we know that $f(t+T)=f(t)$. By Lemma \ref{lemma3}, we know immediately that $T=nT_2$ for some $n \in \mathbb N$.

We will now show $$\int_0^{n T_2} f(\tau)\cos(\omega_0 \tau) d\tau = \int_0^{n T_2} f(\tau)\sin(\omega_0 \tau) d\tau = 0.$$ 
Enforcing the initial conditions, we know that $x(0) = 0$ and $\dot x(0)=0$. Moreover, since $x$ is periodic with period $T$ and we have $T=nT_2$, we have $0=x(0) = x(nT_2)$ and  $0 = \dot x(0) = \dot x(nT_2)$. Expanded in a $2\times2$ system, we have
\begin{align*}
    0 =&~ \frac{\sin(\omega_0 n T_2)}{\omega_0} \int_0^{n T_2}\cos(\omega_0\tau)f(\tau)d\tau - \frac{\cos(\omega_0 n T_2)}{\omega_0} \int_0^{n T_2}\sin(\omega_0\tau)f(\tau)d\tau, \\
       0 =&~ \cos(\omega_0 n T_2) \int_0^{n T_2}\cos(\omega_0\tau)f(\tau)d\tau + \sin(\omega_0 n T_2) \int_0^{n T_2}\sin(\omega_0\tau)f(\tau)d\tau.
\end{align*}
Solving for the integral terms, we must check the invertibility of the matrix
\begin{equation}\label{matrix}
    \begin{bmatrix}
\dfrac{\sin(\omega_0 nT_2)}{\omega_0} &  \dfrac{-\cos(\omega_0 n T_2)}{\omega_0} \\
\cos(\omega_0 n T_2) & \sin(\omega_0 nT_2)
\end{bmatrix}
\end{equation}
(This is equivalent to verifying that the Wronskian  $W\left[\sin(\omega_0t),\cos(\omega_0t)\right]$ does not vanish at $t=nT_2$.)  But $W(t)= \omega_0^{-1}\neq 0$ here, and thus the unique solution is the zero vector, demonstrating that
 \begin{equation}\label{neednow} \int_0^{n T_2}\cos(\omega_0\tau)f(\tau)d\tau = \int_0^{n T_2}\sin(\omega_0\tau)f(\tau)d\tau = 0.\end{equation}
These identities will be central in obtaining our main contradiction below.

We will now show that $\int_0^t f(\tau)\cos(\omega_0 t)d\tau = \int_0^t f(\tau)\sin(\omega_0)d\tau = 0$ for all $t \in [0,\infty)$. Recall that $x(t)=x(t+T)$ for all $t \in [0,\infty)$, but since we have shown that $T=nT_2$ for some $n \in \mathbb N$, we know 
\begin{equation*}\begin{cases}
x(t)-x(t+nT_2)=&~0, \\
\dot x(t)-\dot x(t+nT_2)=&~0,
\end{cases}
\end{equation*}
for all $t \in [0,\infty)$, since the $x, \dot x \in C([0,\infty)).$ These two equations will produce another invertible matrix system which will yield our desired conclusion. 
Using the first equation above, we expand  $\omega_0 x(t+nT_2) = \omega_0 x(t)$ via the variation of parameters formulate to obtain 
\begin{align}\label{laterreference}
     \sin(\omega_0& (t + nT_2)) \int_0^{t + n T_2} f(\tau) \cos(\omega_0 \tau) d\tau - \cos(\omega_0 (t + n T_2)) \int_0^{t + n T_2} f(\tau) \sin(\omega_0 \tau) d\tau \nonumber \\
     = & ~
    \sin(\omega_0 t) \int_0^{t} f(\tau) \cos(\omega_0 \tau) d\tau - \cos(\omega_0 t) \int_0^{t} f(\tau) \sin(\omega_0 \tau) d\tau.
\end{align}
Breaking the first integrals as $(0,t+nT_2)=(0,t)\cup(t,t+nT_2)$, we obtain
\begin{align}\nonumber
\sin(\omega_0& (t+n T_2)) \int_t^{t+n T_2} f(\tau) \cos(\omega_0 \tau) d\tau - \cos(\omega_0 (t+n T_2)) \int_t^{t + n T_2} f(\tau) \sin(\omega_0 \tau) d\tau \\ 
= &~\nonumber
     -\Big(\sin(\omega_0 (t + n T_2))- \sin(\omega_0 t)\Big) \int_0^{t} f(\tau) \cos(\omega_0 \tau) d\tau \\
      &  + \Big(\cos(\omega_0 (t + n T_2)) - \cos(\omega_0 t)\Big) \int_0^{t} f(\tau) \sin(\omega_0 \tau) d\tau. \label{willbeusing}
\end{align}
We now consider the LHS above, and utilize the sine subtraction formula, as well as the Duhamel representation of the unique solution:
\begin{align}\nonumber
\sin(\omega_0& (t+n T_2)) \int_t^{t+n T_2} f(\tau) \cos(\omega_0 \tau) d\tau - \cos(\omega_0 (t+n T_2)) \int_t^{t + n T_2} f(\tau) \sin(\omega_0 \tau) d\tau \\ = &~\nonumber \int_t^{t+n T_2} f(\tau) \sin(\omega_0 (t + n T_2 -\tau)) d\tau \\ = &~\nonumber 
 \int_0^{t+n T_2} f(\tau) \sin(\omega_0 (t + n T_2 -\tau)) d\tau -  \int_0^t f(\tau) \sin(\omega_0 (t + n T_2 -\tau)) d\tau \\ =&~\omega_0x(t+n T_2) -  \int_0^t f(\tau) \sin(\omega_0 (t + n T_2 -\tau)) d\tau. \label{willbeusing1}
\end{align}
 We now focus on the integral term in \eqref{willbeusing1}; we break it, make a substitution to, and again invoke the Duhamel solution representation:
\begin{align}\nonumber
\int_0^t f(\tau) \sin(\omega_0 (t + n T_2 -\tau)) d\tau=& ~\left[\int_0^{nT_2}+\int_{n T_2}^t\right] f(\tau) \sin(\omega_0 (t + n T_2 -\tau)) d\tau \\
=&~\int_{0}^{t - n T_2} f(\tau) \sin(\omega_0 (t - \tau)) d\tau + \int_{0}^{n T_2} f(\tau) \sin(\omega_0 (t + n T_2 -\tau)) d\tau \nonumber\\
=&~\omega_0x(t-n T_2) + \int_{0}^{n T_2} f(\tau) \sin(\omega_0 (t + n T_2 -\tau)) d\tau \nonumber\\
=& ~\omega_0x(t-nT_2) +\sin(\omega_0 (t+n T_2)) \int_{0}^{n T_2} f(\tau) \cos(\omega_0 \tau) d\tau\nonumber \\ & - \cos(\omega_0 (t+n T_2)) \int_{0}^{n T_2} f(\tau) \sin(\omega_0 \tau) d\tau, \label{willbeusing3}
\end{align}
where in the last lines we restrict to $t \geq n T_2$, and carry this through the rest of the calculation. 
However, in \eqref{neednow}, we showed that $\int_{0}^{n T_2} f(u) \sin(\omega_0 u) du = \int_{0}^{n T_2} f(u) \cos(\omega_0 u) du = 0.$ Therefore, $$\int_0^t f(\tau) \sin(\omega_0 (t + n T_2 -\tau)) d\tau=\omega_0 x(t-n T_2),$$ and we can combine this identity with \eqref{willbeusing} and \eqref{willbeusing1} to obtain for all $t \ge nT_2$:
\begin{align}\label{laterreference1} \nonumber
  0=   \omega_0x(t+n T_2) -  \omega_0x(t - n T_2) = &~
    -\Big(\sin(\omega_0 (t + n T_2))- \sin(\omega_0 t)\Big) \int_0^{t} f(\tau) \cos(\omega_0 \tau) d\tau \\ &+ \Big(\cos(\omega_0 (t + n T_2)) - \cos(\omega_0 t)\Big) \int_0^{t} f(\tau) \sin(\omega_0 \tau) d\tau.
\end{align}

Using sum-to-product formulas, we obtain the equality
\begin{equation*}
    0 = \cos\Big(\omega_0 (t - \frac{n T_2}{2})\Big)\sin(\omega_0 \frac{n T_2}{2}) \int_0^{t} f(\tau) \cos(\omega_0 \tau) d\tau +\sin\Big(\omega_0 (t - \frac{n T_2}{2})\Big)\sin(\omega_0 \frac{n T_2}{2}) \int_0^{t} f(\tau) \sin(\omega_0 \tau) d\tau.
\end{equation*}
However,  $$\omega_0\frac{n T_2}{2} = \frac{2\pi}{T_1}\frac{n T_2}{2} = \pi  n \alpha \not \in \pi \mathbb Z,$$ since $ T_2 = \alpha T_1$ and we assumed $\alpha \notin \mathbb Q_+$. Then we know that  
 $\sin(\omega_0 \frac{n T_2}{2})=\sin(\pi \alpha n) \neq 0.$ Therefore in the above, we divide by $\sin(\omega_0\frac{n T_2}{2})$ to obtain the identity
\begin{equation}\label{bigone}
    0 = \cos\Big(\omega_0 (t - \frac{n T_2}{2})\Big) \int_0^{t} f(\tau) \cos(\omega_0 \tau) d\tau + \sin\Big(\omega_0 (t - \frac{n T_2}{2})\Big) \int_0^{t} f(\tau) \sin(\omega_0 \tau) d\tau.
\end{equation}

Now, repeating the same analysis for  $\dot x(t) = \dot x(t+n T_2)$, we obtain a similar identity

\begin{equation}\label{bigone1}
 0=    \sin\Big(\omega_0 (t-\frac{n T_2}{2})\Big) \int_0^t f(\tau)\cos(\omega_0 \tau) d\tau - \cos\Big(\omega_0 (t-\frac{n T_2}{2})\Big) \int_0^t f(\tau) \sin(\omega_0 \tau) d\tau.
\end{equation}
We combine the above equalities into another system:
\begin{equation}
    \begin{bmatrix}
\cos(\omega_0 (t-\frac{n T_2}{2})) & \sin(\omega_0 (t-\frac{n T_2}{2})) \\
\sin(\omega_0 (t-\frac{n T_2}{2})) & - \cos(\omega_0 (t-\frac{n T_2}{2})) 
\end{bmatrix}
\begin{bmatrix}
    \int_0^t f(\tau)\cos(\omega_0 \tau) d\tau \\
    \int_0^t f(\tau) \sin(\omega_0 \tau) d\tau
\end{bmatrix}
 = 
 \begin{bmatrix}
     0 \\
     0
 \end{bmatrix}
\end{equation}
Again, the matrix above is invertible, and hence, for all $t \ge nT_2$, that 
$$\int_0^t f(\tau)\cos(\omega_0 \tau) d\tau = \int_0^t f(\tau) \sin(\omega_0 \tau) d\tau = 0.$$
But, by the Duhamel formula
\begin{equation*}x(t) = \frac{\sin(\omega_0 t)}{\omega_0} \int_0^t f(\tau)\cos(\omega_0 \tau)d\tau - \frac{\cos(\omega_0 t)}{\omega_0} \int_0^t f(\tau)\sin(\omega_0 \tau)d\tau\end{equation*}
we see that $x(t) \equiv 0$ for all $t \ge nT_2$. On the other hand, the assumed periodicity of $x(t+nT_2)=x(t)$ for all $t \in [0,\infty)$, gives that $x(t) \equiv 0$. But this violates that $x$ has period $T>0$ (and subsequently that $f$ has period $T_2$). Hence $x$ cannot be periodic with period $T>0$.
\end{proof}

\subsection{Proof of (2.) in Theorem \ref{th:main}}
We now assume that $T_2=\dfrac{m}{n}T_1$ in lowest terms for $m \in \mathbb N$ and $n \in \mathbb N_{\ge 2}$, with $\frac{m}{n} \notin \mathbb N$. We will show that the solution $x(\cdot) \in L^{\infty}(\mathbb R_+)$  is periodic with period $T_3=nT_2=mT_1$. 
We begin with two observations resulting from Lemmas \ref{lemma1} and \ref{lemma2}. For  $T_2=\dfrac{m}{n}T_1$, we observe by Lemma \ref{lemma1} that
 \begin{equation}\label{need1}\sum_{k=0}^{n-1} \sin\left(2\pi \frac{m}{n}k\right) = \dfrac{\sin(2\pi \frac{m}{n} (n-1)/2)\sin( \pi m)}{\sin(\pi \frac{m}{n})},\end{equation}
 and by Lemma \ref{lemma2} that
\begin{equation}\label{need2} \sum_{k=0}^{n-1} \cos\left(2\pi \frac{m}{n}k\right) = \frac{\cos(2\pi \frac{m}{n} (n-1)/2)\sin( \pi m)}{\sin(\pi \frac{m}{n})}.\end{equation}
    Since $m \in  \mathbb{N},$ we have that $\sin(\pi m) = 0.$ Therefore, both summations above are identically zero.

We now proceed to show that $x(\cdot)$ is periodic with period $T_3$.
 Again, by the variation of parameters formula with $x(0) = \dot x(0) = 0$, we have
\begin{equation*}x(t+T_3) = \dfrac{\sin(\omega_0(t+T_3))}{\omega_0}\int_0^{t+T_3}\cos(\omega_0\tau)f(\tau)d\tau-\dfrac{\cos(\omega_0(t+T_3))}{\omega_0}\int_0^{t+T_3}\sin(\omega_0\tau)f(\tau)d\tau.\end{equation*}

Since $T_3 = nT_2 = m T_1$, and $T_1=\frac{2\pi}{\omega_0}$, we have that 
$$\sin(\omega_0 t) = \sin(\omega_0 (t+T_1)) = \sin(\omega_0 (t+2T_1)) = ... = \sin(\omega_0 (t+mT_1))$$ as well as 
$$\cos(\omega_0 t) = \cos(\omega_0 (t+T_1)) = \cos(\omega_0 (t+2T_1)) = ... = \cos(\omega_0 (t+mT_1)).$$ 

The identity for $x(t+T_3)$ then becomes
\begin{equation} x(t+T_3) = \dfrac{\sin(\omega_0 t)}{\omega_0}\int_0^{t+T_3}\cos(\omega_0\tau)f(\tau)d\tau-\dfrac{\cos(\omega_0 t)}{\omega_0}\int_0^{t+T_3}\sin(\omega_0\tau)f(\tau)d\tau.\end{equation}
We now break the integrals as
\begin{align*} x(t+T_3) =&~ \dfrac{\sin(\omega_0 t)}{\omega_0}\left(\int_0^{t}\cos(\omega_0\tau)f(\tau)d\tau +\int_{t}^{t+T_3}f(\tau)\cos(\omega_0\tau)d\tau\right) \\&-\dfrac{\cos(\omega_0 t)}{\omega_0}\left(\int_0^{t}\sin(\omega_0\tau)f(\tau)d\tau+\int_{t}^{t+T_3}f(\tau)\sin(\omega_0\tau)d\tau\right).\end{align*}
We demonstrate that the integrals on $(t,t+T_3)$ are zero by differentiating, using the fundamental theorem of calculus
$$\dfrac{d}{dt}\left(\int_{t}^{t+T_3}f(\tau)\cos(\omega_0\tau)d\tau\right)= f(t+T_3)\cos(\omega_0(t+T_3))-f(t)\cos(\omega_0t),$$
$$\dfrac{d}{dt}\left(\int_{t}^{t+T_3}f(\tau)\sin(\omega_0\tau)d\tau\right)= f(t+T_3)\sin(\omega_0(t+T_3))-f(t)\sin(\omega_0t).$$ Since $T_3$ is a coincident period for $T_1$ and $T_2$ (i.e., $nT_2=mT_1=T_3$), we observe that both quantities are zero (and are equal to their values at $t=0$). This implies that both integrals are constant, so we have for all $t \in [0,\infty)$ that
\begin{align*}
\int_{t}^{t+T_3}f(\tau)\cos(\omega_0\tau)d\tau = ~\int_0^{T_3}f(\tau)\cos(\omega_0\tau) d\tau;~~
\int_{t}^{t+T_3}f(\tau)\sin(\omega_0\tau)d\tau =~ \int_0^{T_3}f(\tau)\sin(\omega_0\tau) d\tau
\end{align*}
To complete the proof, we will show that $\ds \int_{0}^{T_3}\sin(\omega_0\tau)f(\tau)d\tau = 0.$ Since $T_3 = n T_2$, we break the integral as
\begin{equation*} \int_0^{T_3}f(\tau)\sin(\omega_0\tau) d\tau = \int_0^{T_2} f(\tau)\sin(\omega_0 \tau)d\tau + \int_{T_2}^{2T_2} f(\tau)\sin(\omega_0 \tau)d\tau + ... + \int_{(n-1)T_2}^{nT_2} f(\tau)\sin(\omega_0 \tau)d\tau.\end{equation*}
Via substitution, we obtain
\begin{equation}  \int_0^{T_3}f(\tau)\sin(\omega_0\tau) d\tau = \sum_{k=1}^{n} \int_{0}^{T_2} f\Big(\tau+(k-1)T_2)\Big)\sin\Big(\omega_0 (\tau+(k-1)T_2)\Big)d\tau.\end{equation}
As $f$ is periodic with assumed period $T_2$, we obtain
\begin{align*} \int_0^{T_3}f(\tau)\sin(\omega_0\tau) d\tau =&~ \int_0^{T_2} f(\tau)\sum_{k=0}^{n-1} \sin(\omega_0(\tau+kT_2)) d\tau \\ =&~ 
\int_0^{T_2} f(\tau)\sum_{k=0}^{n-1}\left[\sin(\omega_0 \tau)\cos(2\pi k \frac{m}{n}) + \cos(\omega_0 \tau)\sin(2\pi k \frac{m}{n})\right]d\tau \\
=&~\int_0^{T_2} f(\tau)\left[\sin(\omega_0 \tau)\sum_{k=0}^{n-1}\cos(2\pi k \frac{m}{n}) + \cos(\omega_0 \tau)\sum_{k=0}^{n-1}\sin(2\pi k \frac{m}{n})\right]d\tau.\end{align*}
But by identities \eqref{need1} and \eqref{need2}, we observe that both sums evaluate to zero, and 
$$\int_0^{T_3}f(\tau)\sin(\omega_0\tau) d\tau =0,$$ as desired. 

The proof that $$\int_0^{T_3}f(\tau)\cos(\omega_0\tau) d\tau =0,$$ proceeds analogously.
Returning to the identity from earlier, we then have
\begin{align*} x(t+T_3) =&~ \dfrac{\sin(\omega_0 t)}{\omega_0}\left(\int_0^{t}\cos(\omega_0\tau)f(\tau)d\tau +\int_{t}^{t+T_3}f(\tau)\cos(\omega_0\tau)d\tau\right) \\&-\dfrac{\cos(\omega_0 t)}{\omega_0}\left(\int_0^{t}\sin(\omega_0\tau)f(\tau)d\tau+\int_{t}^{t+T_3}f(\tau)\sin(\omega_0\tau)d\tau\right) \\
=& ~\dfrac{\sin(\omega_0 t)}{\omega_0}\int_0^{t}\cos(\omega_0\tau)f(\tau)d\tau -\dfrac{\cos(\omega_0 t)}{\omega_0}\int_0^{t}\sin(\omega_0\tau)f(\tau)d\tau = x(t).\end{align*}
Thus we have shown that $x(t)$ is periodic with period at most $T_3$.
As $x(\cdot) \in C([0,T_3])$ is the unique ODE solution on the closed interval, it is necessarily bounded. Moreover, extending periodically to all of $\mathbb R_+$, we observe that $x \in L^{\infty}(\mathbb R_+).$

Now we show the minimality of $T_3$ so that $T_3$ is the period of $x(\cdot)$. This argument mirrors the proof of non-periodicity in (1.). Assume to the contrary that $T \in \mathbb{R_{+}}$ is a period of $x(\cdot)$ with $T < T_3$.    Since,  $ \forall t \in \mathbb{R_{+}},$ we have $x(t) = x(t+T_3),$ we know that, by Lemma \ref{lemma3}, $\exists k \in \mathbb{N}, T_3 = kT.$ Therefore, $T = \frac{1}{k} T_3 = \frac{n}{k} T_2 = \frac{m}{k} T_1.$ We also have that $x(t) = x(t+T)$, and, as argued before, $T$ is a period of $f(\cdot)$; since $T_2$ is the (minimal) period of $f(\cdot)$, again we know from Lemma \ref{lemma3} that $\exists j \in \mathbb{N},$ so that $T = j T_2.$ Therefore, $\frac{n}{k} T_2 = j T_2,$ and we observe $\frac{n}{k}\in \mathbb N$. 
 Since $T < T_3 = kT$, we know $k > 1$. Since $m,n$ are relatively prime and $\frac{n}{k} \in \mathbb{N},$ we know $\frac{m}{k} \notin \mathbb{N}.$

    Since $T$ is an integer multiple of $T_2$, we make an analogous argument to \ref{matrix}. It then follows that \begin{equation}\int_0^{\frac{n}{k} T_2}f(\tau)\cos(\omega_0 \tau) d\tau = 
\int_0^{\frac{n}{k} T_2}  f(\tau)\sin(\omega_0 \tau) d\tau = 0.\end{equation}
We want to show $\int_0^t f(\tau)\sin(\omega_0 \tau) d\tau = \int_0^t f(\tau)\cos(\omega_0 \tau) d\tau = 0.$ 
We repeat the lines \ref{laterreference}--\ref{laterreference1}, replacing $nT_2$ with $\frac{n}{k} T_2$. Therefore, we obtain, via the sum-to-product formulae, {\small
\begin{equation*}
    0 = \cos\left(\omega_0\left(t-\frac{T}{2}\right)\right)\sin\left(\omega_0 \frac{T}{2}\right)\int_0^t f(\tau)\cos(\omega_0 \tau) d\tau + \sin\left(\omega_0\left(t-\frac{T}{2}\right)\right)\sin\left(\omega_0 \frac{T}{2}\right)\int_0^t f(\tau)\sin(\omega_0 \tau) d\tau
\end{equation*}}
for $t \geq T/2$. Since 
$$\omega_0 \frac{T}{2} = \frac{2\pi}{T_1} \frac{\frac{n}{k} T_2}{2} = \pi \frac{n}{k} \frac{m}{n} = \pi \frac{m}{k},$$
and $\frac{m}{k} \notin \mathbb N$, we know $\sin(\omega_0 \frac{T}{2}) \neq 0.$ Therefore, we divide by this quantity to obtain
\begin{equation}
    0 = \cos(\omega_0(t-\frac{T}{2}))\int_0^t f(\tau)\cos(\omega_0 \tau) d\tau + \sin(\omega_0(t-\frac{T}{2}))\int_0^t f(\tau)\sin(\omega_0 \tau) d\tau
\end{equation}
 for $t \geq T/2$. Similarly to \ref{bigone1}, repeating the same reasoning but with $\dot x(t) = \dot x(t+n T_2)$, we derive
\begin{equation}
    0 = \sin(\omega_0(t-\frac{T}{2}))\int_0^t f(\tau)\cos(\omega_0 \tau) d\tau -\cos(\omega_0(t-\frac{T}{2}))\int_0^t f(\tau)\sin(\omega_0 \tau) d\tau
\end{equation}
with $t \geq T/2$. As before, we obtain the system
\begin{equation}
\begin{bmatrix}
\cos(\omega_0(t-\frac{T}{2})) & \sin(\omega_0(t-\frac{T}{2}))\\
\sin(\omega_0(t-\frac{T}{2})) & -\cos(\omega_0(t-\frac{T}{2}))
\end{bmatrix}
\begin{bmatrix}
\int_0^{t}f(\tau)\cos(\omega_0 \tau) d\tau\\
\int_0^{t}  f(\tau)\sin(\omega_0 \tau) d\tau
\end{bmatrix}
=
\begin{bmatrix}
    0\\
    0
\end{bmatrix}.
\end{equation}
The above matrix is invertible, so $\int_0^{t}f(\tau)\cos(\omega_0 \tau) d\tau = 
\int_0^{t}  f(\tau)\sin(\omega_0 \tau) d\tau = 0.$ Therefore, $x(t) \equiv 0, t \geq T/2.$ Since $x(t) = x(t+T), \forall t \in \mathbb{R_{+}},$ we know $x(t) \equiv 0, \forall t \in \mathbb{R_{+}}.$ This violates the assumption that $x(t)$ has a period $0<T < T_3$. Hence, x has a (minimal) period of $T_3$.

\subsection{Proof of (3.) in Theorem \ref{th:main}}

We now suppose that $\alpha = \frac{m}{n} \in \mathbb{N},$ and  $$(f(t),\cos(\omega_0 t))_{L^2(0,T_3)} = (f(t),\sin(\omega_0 t))_{L^2(0,T_3)} = 0.$$
We will then show that  $x(\cdot) \in L^{\infty}(\mathbb R_+)$ is periodic with period $T_3$. Since $m/n \in \mathbb N$, we know that there is some $l \in \mathbb N$ so that $T_3=T_2=lT_1$. 
We will show that $\forall t \in [0,\infty), x(t)=x(t+T_3).$ 

By the variation of parameters formula, with $x(0) = \dot x(0) = 0,$ we  write
\begin{equation*} x(t+T_3) = \dfrac{\sin(\omega_0 (t+T_3))}{\omega_0}\int_0^{t+T_3}\cos(\omega_0\tau)f(\tau)d\tau  -\dfrac{\cos(\omega_0 (t+T_3))}{\omega_0}\int_0^{t+T_3}\sin(\omega_0\tau)f(\tau)d\tau.\end{equation*}

   Since $lT_1=T_3$, and $\sin(\omega_0t), \cos(\omega_0t)$ are periodic with period $T_1$, we know they are also periodic with period $T_3$, which allows us to simplify again as

\begin{align*}
x(t+T_3) =&~ \dfrac{\sin(\omega_0 t)}{\omega_0}\left(\int_0^{t}\cos(\omega_0\tau)f(\tau)d\tau+\int_t^{t+T_3}\cos(\omega_0\tau)f(\tau)d\tau\right) \\
    & - \dfrac{\cos(\omega_0 t)}{\omega_0}\left(\int_0^{t}\sin(\omega_0\tau)f(\tau)d\tau+\int_t^{t+T_3}\sin(\omega_0\tau)f(\tau)d\tau\right). 
\end{align*}  
As in the previous section, we note that
$$\dfrac{d}{dt}\left(\int_t^{t+T_3}\cos(\omega_0\tau)f\tau d\tau\right)=\cos(\omega_0(t+T_3))f(t+T_3)-\cos(\omega_0t)f(t) =0,$$
since both $\cos(\omega_0t),~f(t)$ are periodic with coincident period $T_3$. 
Hence we know this value is constant, and equal to its value at $t=0$, so 
$$\int_t^{t+T_3}\cos(\omega_0\tau)f(\tau) d\tau= \int_0^{T_3}\cos(\omega_0\tau)f(\tau)d\tau.$$
Similarly, we obtain 
$$\int_t^{t+T_3}\sin(\omega_0\tau)f(\tau) d\tau= \int_0^{T_3}\sin(\omega_0\tau)f(\tau)d\tau.$$
Both of these integral quantities are zero, by hypothesis, and hence we obtain the desired identity
  \begin{align*}
x(t+T_3) =&~ \dfrac{\sin(\omega_0 t)}{\omega_0}\int_0^{t}\cos(\omega_0\tau)f(\tau)d\tau \\
    & - \dfrac{\cos(\omega_0 t)}{\omega_0}\int_0^{t}\sin(\omega_0\tau)f(\tau)d\tau\\ =&~ x(t). 
\end{align*}  
    Since this identity holds for all $t \in [0,\infty),$ $x(\cdot)$ is periodic with period at most $T_3$. As before, since $x\in C([0,T_3])$ as the ODE solution, we know that it is bounded, and hence, its periodic extension to all of $\mathbb R_+$ is also bounded; this gives that $x \in L^{\infty}(\mathbb R_+)$. 
    
    Now we show the minimality of $T_3$ so that $T_3$ is the period of $x(\cdot)$. Assume to the contrary that $T \in \mathbb{R_{+}},$ with $T < T_3$, is a period of $x(\cdot)$. 
    Since,  $ \forall t \in \mathbb{R_{+}},$ we have $x(t) = x(t+T_3),$ we know that, by Lemma \ref{lemma3}, $\exists k \in \mathbb{N}, T_3 = kT.$ Therefore, $T = \frac{1}{k} T_3 = \frac{n}{k} T_2 = \frac{m}{k} T_1.$ We also have that $x(t) = x(t+T)$, and, as argued before, $T$ is a period of $f(\cdot)$; since $T_2$ is the (minimal) period of $f(\cdot)$, again we know from Lemma \ref{lemma3} that $\exists j \in \mathbb{N},$ so that $T = j T_2.$ Therefore, $\frac{n}{k} T_2 = j T_2,$ and we observe $\frac{n}{k}\in \mathbb N$. Now, since $\frac{m}{n} \in \mathbb N$ by hypothesis, we know that $\frac{n}{k}\times\frac{m}{n} = \frac{m}{k}\in \mathbb N$. Therefore, $k$ properly divides both $m$ and $n$. Moreover, since $T < T_3 = \frac{1}{k} T,$ we know $k > 1$, which yields that $\gcd(m,n) \geq k > 1$ contradicting that $\alpha = m/n \in \mathbb N$ was in lowest terms. 
   This contradiction yields that $T_3 \leq T$, and $T_3$ is the (minimal) period of $x(\cdot)$.
    
\subsection{Proof of (4.) in Theorem \ref{th:main}}
We now make the same assumption that  $\alpha = \frac{m}{n} \in \mathbb{N},$ and so there is $l \in \mathbb N$ so that $T_3=T_2=lT_1$. 
However, we will assume that $(f(t),\cos(\omega_0 t))_{L^2(0,T_3)} \neq 0$ or $(f(t),\sin(\omega_0 t))_{L^2(0,T_3)} \neq 0$. We will give the proof for $(f(t),\sin(\omega_0 t))_{L^2(0,T_3)} \neq 0$, of course not making any assumptions about the value of $(f(t),\sin(\omega_0 t))_{L^2(0,T_3)}$. The proof is then easily modified for the other case. We will show that  $x(t)$ is resonant when  $(f(t),\sin(\omega_0 t))_{L^2(0,T_3)} \neq 0$; in particular, given a value $L \in \mathbb R_+$, we will produce a time $t_1(L)$ so that $|x(t_1)|>L$. 

To that end, let $L \in \mathbb R_+$. Also, denote  $Q_1 = (f(t),\sin(\omega_0 t))_{L^2(0,T_3)}$. 
Then take $$t_1 = T_3 \left(\left\lfloor \frac{\omega_0 L}{|Q_1|} \right\rfloor+2\right)>0,$$ where we utilize the standard floor function.
Invoking the variation of parameters formula, we have
\begin{equation}x(t_1) =\dfrac{\sin(\omega_0t_1)}{\omega_0}\int_0^{t_1}\cos(\omega_0\tau)f(\tau)d\tau-\dfrac{\cos(\omega_0t_1)}{\omega_0}\int_0^{t_1}\sin(\omega_0\tau)f(\tau)d\tau.\end{equation} 
Since $T_3=T_2=lT_1=\frac{2\pi l}{\omega_0}$ we have
$$\sin(\omega_0 t_1) = \sin\left(\omega_0 T_3 \left(\left\lfloor \frac{\omega_0 L}{|Q_1|} \right\rfloor+2\right)\right)=\sin\left(2\pi l\left(\left\lfloor \frac{\omega_0 L}{|Q_1|} \right\rfloor+2\right)\right).$$
Above, $l\lfloor \frac{\omega_0 L}{|Q_1|} \rfloor+2 \in \mathbb N$, so $\sin(\omega_0t_1)= 0$ and $\cos(\omega_0t_1)=1$. Therefore, the expression for $x(t)$  becomes 
\begin{equation}x(t_1) = -\dfrac{1}{\omega_0}\int_0^{t_1}\sin(\omega_0\tau)f(\tau)d\tau.\end{equation}
Since we know that $t_1 = T_3(\lfloor \frac{\omega_0 L}{|Q_1|} \rfloor+2)$ and $\lfloor \frac{\omega_0 L}{|Q_1|} \rfloor+2 \in \mathbb N$, we break the integral accordingly. 
\begin{equation*}
  \int_0^{t_1}\sin(\omega_0\tau)f(\tau)d\tau=   \int_0^{T_3} f(\tau)\sin(\omega_0 \tau)d\tau + \int_{T_3}^{2T_3} f(\tau)\sin(\omega_0 \tau)d\tau + ... + \int_{(\lfloor \frac{\omega_0 L}{|Q_1|} \rfloor+1)T_3}^{(\lfloor \frac{\omega_0 L}{|Q_1|} \rfloor+2)T_3} f(\tau)\sin(\omega_0 \tau)d\tau.
\end{equation*}
As we have done in previous sections, we use substitution and the coincident period $T_3=T_2=lT_1$, to write
\begin{equation}
     \int_0^{t_1}\sin(\omega_0\tau)f(\tau)d\tau= \left(\left\lfloor \frac{\omega_0 L}{|Q_1|} \right\rfloor+2\right)\int_0^{T_3} f(\tau)\sin(\omega_0 \tau)d\tau.
\end{equation}

Since we have denoted $Q_1 = (f(t),\sin(\omega_0 t))_{L^2[0,T_3]} = \int_0^{T_3}\sin(\omega_0\tau)f(\tau)d\tau,$ 
we have
\begin{equation}
 |x(t_1)|  =\dfrac{1}{\omega_0} \left|  \int_0^{t_1}\sin(\omega_0\tau)f(\tau)d\tau\right|= \dfrac{1}{\omega_0}\left(\left\lfloor \frac{\omega_0 L}{|Q_1|} \right\rfloor+2\right)|Q_1|.
\end{equation}
But it follows from the definition of the floor function that
\begin{equation}0< L < |x(t_1)|\end{equation}
as desired.

The case where $Q_2 \equiv (f(t),\cos(\omega_0 t))_{L^2(0,T_3)} \neq 0$ is handled by making a modified choice of $t_1$. With $L>0$, $a = \int_0^{\frac{T_1}{4}} f(\tau)\cos(\omega_0 \tau) d\tau$, and $\lceil \cdot \rceil$ denoting the ceiling function, we define
$$t_1 = T_3\left(\left\lceil \frac{\omega_0 L + |a|}{|Q_2|} \right\rceil + 1\right) + \frac{T_1}{4} = \frac{2\pi l}{\omega_0}\left(\left\lceil \frac{\omega_0 L + |a|}{|Q_2|} \right\rceil + 1\right) + \frac{\pi}{2\omega_0}. $$ 

In this case, we have that $\omega_0t_1-\frac{\pi}{2} \in 2\pi \mathbb N$. In this case, 
$$\cos(\omega_0t_1)=0,~~\sin(\omega_0t_1)=1,$$
and the proof that $|x(t_1)|>L$ proceeds similarly as in the previous case. We know $\left| x(t_1) \right|$ can be represented as

\begin{equation}
    |x(t_1)| = \left|\frac{1}{\omega_0} \int_0^{t_1} f(\tau)\cos(\omega_0 \tau) d\tau\right|.
\end{equation}
 As with the previous case, we  use substitution and the coincident period to write the integral above as

\begin{equation}
    \left[
    \int_0^{T_3\left(\left\lceil \frac{\omega_0 L + |a|}{|Q_2|} \right\rceil + 1\right)} 
     + \int_{T_3\left(\left\lceil \frac{\omega_0 L + |a|}{|Q_2|} \right\rceil + 1\right)}^{T_3\left(\left\lceil \frac{\omega_0 L + |a|}{|Q_2|} \right\rceil + 1\right) + \frac{T_1}{4}} \right] f(\tau)\cos(\omega_0 \tau) d\tau = \left(\left\lceil \frac{\omega_0 L + |a|}{|Q_2|} \right\rceil + 1\right) Q_2 + a.
\end{equation}

It follows from the definition of the ceiling function that $L < \frac{|Q_2|}{\omega_0} \left(\left\lceil \frac{\omega_0 L + |a|}{|Q_2|} \right\rceil + 1\right) - \frac{|a|}{\omega_0}.$ This then leads us to the conclusion that

\begin{equation}
    L < \left| \frac{|Q_2|}{\omega_0} \left(\left\lceil \frac{\omega_0 L + |a|}{|Q_2|} \right\rceil + 1\right) - \frac{|a|}{\omega_0} \right| \leq \left| \frac{1}{\omega_0} \left(\left\lceil \frac{\omega_0 L + |a|}{|Q_2|} \right\rceil + 1\right) Q_2 + \frac{a}{\omega_0} \right| = |x(t_1)|.
\end{equation}

\section{Appendix: Laplace Transform Approach}
We recall the classical Laplace transform for exponentially bounded functions \cite{farlow}. Namely, that $f: [0,\infty)\to \mathbb R$ is locally integrable, and suppose there exists constants $a,C\ge 0$ so that
$$|f(t)| \le Ce^{at},~~\forall~t \in [0,\infty).$$ Then we define $\mathcal L[f](s) \equiv F(s)$ as
$$F(s) \equiv \int_0^{\infty}e^{-st}f(t)ds,~~\forall ~s \in [a,\infty).$$
A standard topic in undergraduate ODE courses is the use of the Laplace transform to solve second order linear ODEs with forcing, including those with discontinuous forcing functions. The transform is applied to the Cauchy problem, and once an algebraic expression is found for $X(s)$ (the solution), this expression simply needs to be inverted. In general, this is a challenging process---see {\em Mellin's Formula} or the discussion in an elementary ODE textbook \cite{farlow,BdP}. 

For the oscillator system of interest here, 
\begin{equation} \begin{cases}
\ddot{x}+\omega_0^2x=f(t),  \\
x(0)=0;~~\dot{x}(0)=0.
\end{cases}
\end{equation}
We can take the Laplace transform of both sides, resulting in the algebraic expression
$$s^2X(s)+\omega_0^2X(s)=F(s)~~\implies~~X(s)=\dfrac{F(s)}{s^2+\omega_0^2}.$$
In order to obtain an expression for the unique solution, or at least to analyze its properties, it is necessary to invert 
$$x(t)=\mathcal L^{-1}\left[\dfrac{F(s)}{s^2+\omega_0^2}\right](t),$$
when possible. 

In the case that $f$ is periodic, we note a helpful identity that aids in the inversion process. Using an argument akin to those presented in the proof of the main theorem, we have: 
\begin{lemma}
Suppose that $f: \mathbb [0,\infty) \to \mathbb R$ is bounded, locally-integrable, and periodic with period $T>0$. Then 
$$F(s)=\mathcal L[f](s) = \dfrac{\displaystyle \int_0^Te^{-st}f(t)dt}{1-e^{-sT}}.$$
\end{lemma}
\begin{proof}
We consider the integral defining the transform and break it at multiples of the period of $f$:
\begin{equation}
\int_0^{\infty}e^{-st}f(t)dt = \sum_{k=0}^{\infty} \int_{kT}^{(k+1)T}e^{-st}f(t)dt.
\end{equation}
Let $u = t-kT$ and change variable in the integrand to obtain
$$\int_{kT}^{(k+1)T}e^{-st}f(t)dt = \int_0^Te^{-s(u+kT)}f(u+kT)du=e^{-skT}\int_0^Te^{-su}f(u)du,$$
using the $T$-periodicity of $f$. 
Thus, we may rewrite the desired integral as a geometric series, which yields the result.
\begin{equation}
\int_0^{\infty}e^{-st}f(t)dt =  \left(\int_{0}^{T}e^{-st}f(t)dt\right)\sum_{k=0}^{\infty}e^{-skT} =\dfrac{\ds\int_{0}^{T}e^{-st}f(t)dt}{1-e^{-sT}}.
\end{equation}
\end{proof}
Using the above lemma, we can convert the question of resonance into a simple statement: Given 
$$x(t) = \mathcal L^{-1}\left[ \frac{ \int_0^Te^{-st}f(t)dt}{(1-e^{-sT})(s^2+\omega_0^2)} \right] (t),$$
we have that $x(t)$ is {\em not resonant} if and only if $\mathcal L^{-1}\left[ \dfrac{ \int_0^Te^{-st}f(t)dt}{(1-e^{-sT})(s^2+\omega_0^2)} \right]  \in L^{\infty}([0,\infty)).$ Of course, as mentioned above, inversion formulae for the Laplace transform are non-trivial; on the other hand, this expression provides a check that does not directly involve a numerical ODE solve.

\end{document}